\documentclass[letterpaper, 
	abstract = true, 
	10pt]{scrartcl}

\usepackage[utf8]{inputenc}

\usepackage{amsthm, amssymb, amsmath, mathtools, standalone, tikz}

\usepackage[affil-sl]{authblk}

\usepackage{outlines}

\usepackage[colorinlistoftodos, disable]{todonotes}

\usepackage[headings]{fullpage}

\usepackage{setspace}
\DeclareMathAlphabet{\mathbbold}{U}{bbold}{m}{n}

\usepackage{scrlayer-scrpage}
\renewcommand\pagemark{{\usekomafont{pagenumber}\thepage\ of \pageref*{LastPage}}}
\usepackage{pageslts} %
\setkomafont{pageheadfoot}{\rmfamily}
\lohead*{}
\cohead*{}
\lofoot*{}
\cofoot*{}
\rofoot*{\pagemark}

\usepackage[%
  backend=biber,%
  style=numeric,%
  natbib=true,%
  url=false,%
  doi=true,%
  eprint=true,%
  backref=true,%
  maxbibnames=99,%
  sorting=nyt, %
]{biblatex}

\usepackage[unicode, colorlinks=true, linkcolor=blue, citecolor=blue, bookmarks, hyperindex, hyperfigures]{hyperref}

\DeclareFieldFormat{eprint:mr}{%
  \space \textsc{mr}\addcolon\space
  \ifhyperref
    {\href{http://www.ams.org/mathscinet-getitem?mr=MR#1}{\nolinkurl{#1}}}
    {\nolinkurl{#1}}}

\DeclareSourcemap{
  \maps[datatype=bibtex]{
    \map{
      \step[fieldsource=mrnumber, fieldtarget=eprint, final]
      \step[fieldset=eprinttype, fieldvalue=mr]
    }
  }
}

\DeclareFieldFormat{eprint:zbmath}{%
  \space \textsc{zbl}\addcolon\space
  \ifhyperref
    {\href{http://zbmath.org/?q=an:#1}{\nolinkurl{#1}}}
    {\nolinkurl{#1}}}

\DeclareSourcemap{
  \maps[datatype=bibtex]{
    \map{
      \step[fieldsource=zbmath, fieldtarget=eprint, final]
      \step[fieldset=eprinttype, fieldvalue=zbmath]
    }
  }
}
  
\renewbibmacro{in:}{%
  \ifentrytype{article}{}{%
  \printtext{\bibstring{in}\intitlepunct}}}
\DeclareFieldFormat[inproceedings]{pages}{\lowercase{pp.~}#1}

\newtheoremstyle{plain}
  {3mm}                         %
  {3mm}                         %
  {\slshape}                    %
  {}                            %
  {\bfseries}       %
  {.}                           %
  {.5em}                        %
  {}                            %

\newtheoremstyle{definition}
  {3mm}                         %
  {3mm}                         %
  {}                            %
  {}                            %
  {\bfseries}       %
  {.}                           %
  {.5em}                        %
  {}                            %

\theoremstyle{plain}

\newtheorem{theorem}{Theorem}[section]
\newtheorem{lemma}[theorem]{Lemma}
\newtheorem{proposition}[theorem]{Proposition}
\newtheorem{corollary}[theorem]{Corollary}

\theoremstyle{definition}
\newtheorem{definition}[theorem]{Definition}
\newtheorem{remark}[theorem]{Remark}
\newtheorem{example}[theorem]{Example}

\newtheorem{question}[theorem]{Question}
\newtheorem{problem}[theorem]{Problem}

\newcommand{\define}[1]{{\textsl{\textbf{#1}}}}
\newcommand{\f}[1]{\mathcal{#1}}

\author[1]{Shea D.~Burns}
\author[2]{Dennis Davenport}
\author[2]{Shakuan Frankson}
\author[3]{Conner Griffin}
\author[3]{John H.~Johnson Jr.}
\author[2]{Malick Kebe}

\affil[1]{%
	\small
	Department of Mathematics\\
	North Carolina A\&T State University\\
	Greensboro, NC%
}

\affil[2]{%
	\small
	Department of Mathematics\\
	Howard University University\\
	Washington, DC
}

\affil[3]{%
	\small
	Department of Mathematics\\
 	The Ohio State University\\
	Columbus, OH
}

\affil[ ]{%
	\small
	\texttt{\href{mailto:sburns@ncat.edu}{sburns@ncat.edu}, \href{mailto:dennis.davenport@live.com}{dennis.davenport@live.com}, \href{mailto:shakuan.frankson@bison.howard.edu}{shakuan.frankson@bison.howard.edu}, \href{mailto:griffin.1101@osu.edu}{griffin.1101@osu.edu}, \href{mailto:johnson.5316@osu.edu}{johnson.5316@osu.edu}, \textrm{and} \href{mailto:malick.kebe@bison.howard.edu}{malick.kebe@bison.howard.edu}}
}

\title{Characterizing \( (\mathcal{F}, \mathcal{G}) \)-syndetic, \( (\mathcal{F}, \mathcal{G}) \)-thick, and related notions of size using derived sets along ultrafilters}
\date{}

\begin{document}
\maketitle

\begin{abstract}
    We characterize relative notions of syndetic and thick sets using what we call ``derived'' sets along ultrafilters.
    Manipulations of derived sets is a characteristic feature of algebra in the Stone--Čech compactification of discrete semigroups and its applications.
    Combined with the existence of idempotents and structure of the smallest ideal in closed subsemigroups of the Stone--Čech compactification, our particular use of derived sets adapt and generalize methods recently used by Griffin \cite{Griffin:2024aa} to characterize relative piecewise syndetic sets.
    As an application, we define an algebraically interesting subset of the Stone--Čech compactification and show, in some ways, it shares structural properties analogous to the smallest ideal.

	\vspace{2em}

	\vspace{2em}

\end{abstract}

\section{Introduction}
\label{section:introduction}

In this paper, we generalize methods using what we call ``derived'' sets along an ultrafilter to characterize certain ``relative'' notions of (piecewise) syndetic sets, thick sets, and two more related notions of sizes.
Recently, Griffin used derived sets to positively answer \cite[Theorem~4.10]{Griffin:2024aa} a question posed by Christopherson and Johnson \cite[Question~4.6]{Christopherson:2022wr} on whether Shuungula, Zelenyuk, and Zelenyuk's notion of relative piecewise syndetic \cite{Shuungula:2009ty} is equivalent to another notion in \cite[Definition~4.1]{Christopherson:2022wr}. 
Building on Griffin's method, a major theme of this paper is that derived sets provide an elegant organizing principle to understand and manipulate relative notions of size: for instance one interpretation of our Corollary~\ref{corollary:relative-syndetic-thick}, an immediate consequence of one of our main results (Theorem~\ref{theorem:relative-syndetic-thick}), is relative syndetic and thick sets inevitably arise when considering derived sets along ultrafilters.

Manipulations of derived sets are foundational in algebra of the Stone--Čech compactification and its applications.
Prototypical examples include the Galvin--Glazer idempotent ultrafilter proof \cite[Sec.~5.2]{Hindman:2012tq} of Hindman's finite sums theorem \cite{Hindman1974}; Papazyan's proof showing all idempotent filters extend to an idempotent ultrafilter  \cite{Papazyan1989};
Bergelson and Hindman, with the assistance of B.~Weiss, minimal idempotent ultrafilter proof \cite{Bergelson:1990aa} of Furstenberg's central sets theorem \cite[Proposition~8.21]{Furstenberg1981a}, along with the generalization of this theorem by De, Hindman, and Strauss \cite{De:2008aa}; and
Beiglböck's ultrafilter proof \cite{Beiglbock:2011aa} of a generalization of Jin's sumset theorem \cite{Jin:2002aa}.
This paper's main focus is on the relationship of derived sets to generalizations of the three classical notions of syndetic, thick, and piecewise syndetic sets in a semigroup.
For instance, see Hindman's survey \cite{Hindman2019a} for an historical overview of these, and other, notions of size.

In the following definition, and in the rest of this paper, given a set \( X \) we let \( \mathcal{P}_f(X) \) denote the collection of all finite nonempty subsets of \( X \).
If \( (S, \cdot) \) is a semigroup, \( A \subseteq S \), and \( h \in S \), then we put \( h^{-1}A := \{ x \in S : hx \in A \} \).
\begin{definition}
\label{definition:(piecewise-)syndetic-and-thick}    
    Let \( (S, \cdot) \) be a semigroup and let \( A \subseteq S \).
    \begin{itemize}
        \item[(a)]
            \( A \) is \define{syndetic} if and only if there exists \( H \in \mathcal{P}_f(S) \) such that \( \bigcup_{h \in H} h^{-1}A = S \).

        \item[(b)]
            \( A \) is \define{thick} if and only if for all \( H \in \mathcal{P}_f(S) \) we have \( \bigcap_{h \in H} h^{-1}A \ne \emptyset \).

        \item[(c)]
            \( A \) is \define{piecewise syndetic} if and only if there exist syndetic \( B \subseteq S \) and thick \( C \subseteq S \) with \( A = B \cap C \).

        \item[(d)] We define the following collections:
        \begin{itemize}
            \item[(i)] \( \mathsf{Syn} := \{ A \subseteq S : A \text{ is syndetic} \} \).
            \item[(ii)] \( \mathsf{Thick} := \{ A \subseteq S : A \text{ is thick} \} \).
            \item[(iii)] \( \mathsf{PS} := \{ A \subseteq S : A \text{ is piecewise syndetic} \} \).
        \end{itemize}
    \end{itemize}
\end{definition}
In particular for the positive integers \( (\mathbb{N}, +) \), a set is syndetic if and only if it has bounded gaps; a set is thick if and only if it contains arbitrarily long blocks of consecutive positive integers; and, a set is piecewise syndetic if and only if there is a fixed bound such that the set contains arbitrarily long subsets of positive integers whose gaps are no bigger than the bound.

In Section~\ref{section:review-algebraic-structure} we provide a brief overview of the algebraic structure of the Stone--Čech compactification.
But we note syndetic, thick, and piecewise syndetic sets have  elegant characterizations, due to Bergelson, Hindman, and McCutcheon \cite{Bergelson1998}, in terms of this algebraic structure. %
For instance, the following known characterization of piecewise syndetic sets in terms of the smallest ideal \( K(\beta S) \) is the motivating theorem for one of our main results: 
\begin{theorem}
\label{theorem:characterization-piecewise-syndetic}
    Let \( (S, \cdot) \) be a semigroup and let \( A \subseteq S \).
    The following are equivalent.
    \begin{itemize}
        \item[(a)]
            \( A \in \mathsf{PS} \).

        \item[(b)]
            There exists \( H \in \mathcal{P}_f(S) \) such that \( \bigcup_{h \in H} h^{-1}A \in \mathsf{Thick} \).

        \item[(c)]
            There exists a minimal left ideal \( L \subseteq \beta S \) such that \( \overline{A} \cap L \ne \emptyset \).

        \item[(d)]
            \( \overline{A} \cap K(\beta S) \ne \emptyset \).

        \item[(e)]
            There exists \( q \in K(\beta S) \) such that \( \{ h \in S : h^{-1}A \in q \} \in \mathsf{Syn} \).
        
         \item[(f)]
          There exists \( e \in E(K(\beta S)) \) such that \( \{ h \in S : h^{-1}A \in e \} \in \mathsf{Syn} \).
         
         \item[(g)]
            There exists \( e \in E(\beta S) \) such that \( \{ h \in S : h^{-1}A \in e \} \in \mathsf{Syn} \).

         \item[(h)]            
            There exists \( q \in \beta S \) such that \( \{ h \in S : h^{-1}A \in q \} \in \mathsf{Syn} \).            
    \end{itemize}
\end{theorem}
\begin{proof}
    \textbf{(a) \( \Leftrightarrow \) (b):} This equivalence is a proved in \cite[Theorem~4.49]{Hindman:2012tq}.

    \smallskip

    \textbf{(b) \( \Rightarrow \) (c):}
    Pick \( H \in \mathcal{P}_f(S) \) as guaranteed for \( A \).
    Then because \( \bigcup_{h \in H} h^{-1}A \) is thick, by \cite[Theorem~4.48(a)]{Hindman:2012tq}, it follows we can pick a minimal left ideal \( L \) of \( \beta S \) with \( L \subseteq \overline{\bigcup_{h \in H} h^{-1}A} \).
    Let \( q \in L \).
    Then \( \bigcup_{h \in H} h^{-1}A \in q \) and we can pick \( h \in H \) with \( h^{-1}A \in q \), that is, \( A \in h \cdot q \).
    Since \( L \) is left ideal we have \( h \cdot q \in L \) and so \( \overline{A} \cap L \ne \emptyset \).
    
    \smallskip

    \textbf{(c) \( \Leftrightarrow \) (d):}
    The smallest ideal \( K(\beta S) \) is the union of all minimal left ideals of \( \beta S \).

    \smallskip

    \textbf{(d) \( \Rightarrow \) (e):}
    Pick \( q \in \overline{A} \cap K(\beta S) \) as guaranteed.
    By \cite[Theorem~4.39]{Hindman:2012tq} we have \( \{ h \in S : h^{-1}A \in q \} \in \mathsf{Syn} \).

    \smallskip

    \textbf{(e) \( \Rightarrow \) (f):}
    Pick \( q \in K(\beta S) \) as guaranteed for \( A \). 
    Let \( L \) be a minimal left ideal of \( \beta S \) with \( q \in L \) and let \( e \in E(L) \).
    By minimality we have \( L= L \cdot q \).

    Now from \cite[Theorem~4.48(b)]{Hindman:2012tq} we have \( \{ h \in S : h^{-1}A \in e \} \) is syndetic if and only if \( \overline{\{ h \in S : h^{-1}A \in e \}} \) intersects every left ideal of \( \beta S \). 
    This is equivalent to showing \( \overline{\{ h \in S : h^{-1}A \in e \}} \cap \beta S \cdot p \ne \emptyset \) for all \( p \in \beta S \).
    
    Let \( p \in \beta S \).
    Then \( p \cdot e \in L = L \cdot q \) and we can pick \( u \in L \) with \( p \cdot e = u \cdot q \).
    By assumption (and \cite[Theorem~4.48(b)]{Hindman:2012tq} again) we have \( \overline{\{ h \in S : h^{-1}A \in q \}} \cap \beta S \cdot u \ne \emptyset \).
    Pick \( r \in \beta S \) with \( \{ h \in S : h^{-1}A \in q \} \in r \cdot u \), that is, \( A \in r \cdot u \cdot q = r \cdot p \cdot e \).
    Hence \( \{ h \in S : h^{-1}A \in e \} \in r \cdot p \), that is, \( \overline{\{ h \in S : h^{-1}A \in e \}} \cap \beta S \cdot p \ne \emptyset \).
    Therefore \( \{ h \in S : h^{-1}A \in e \} \in \mathsf{Syn} \).
    
   \smallskip

   \textbf{(f) \( \Rightarrow \) (g) and (g) \( \Rightarrow \) (h):}
   Both these directions are immediate.

   \textbf{(h) \( \Rightarrow \) (b):}
   Pick \( q \in \beta S \) as guaranteed and let \( B := \{ h \in S : h^{-1}A \in q \} \).
   Pick \( H \in \mathcal{P}_f(S) \) with \( \bigcup_{h \in H} h^{-1}B = S \). 
   We show \( \bigcup_{h \in H} h^{-1}A \) is thick.
   To that end let \( F \in \mathcal{P}_f(S) \) and for each \( f \in F \) pick \( h_f \in H \) with \( h_f f \in B \).
   Therefore \( \bigcap_{f \in F} (h_f f)^{-1}A \in q \) and we can pick \( x \) in this intersection. 
   Then \( x \in \bigcap_{f \in F} f^{-1}\bigl( \bigcup_{h \in H} h^{-1}A \bigr) \).
\end{proof}
Two of our main results (Theorem~\ref{theorem:relative-piecewise-syndetic} and Corollary~\ref{corollary:relative-kernel}) generalizes Theorem~\ref{theorem:characterization-piecewise-syndetic} to relative notions of piecewise syndetic sets (Definition~\ref{definition:relative-piecewise-syndetic}). 
The central method of our proof adapts Griffin \cite{Griffin:2024aa} and involves manipulations of sets of the form \( A'(q) := \{ h \in S : h^{-1}A \in q \} \), a derived set of \( A \) along an ultrafilter \( q \) (Section~\ref{section:derived-sets}).

\subsection*{Organization of article}
\label{section:organization}

In Section~\ref{section:preliminaries} we state more-or-less known or folklore results (Corollary~\ref{corollary:mesh-operator} and Propositions~\ref{proposition:stack}, \ref{proposition:filter-grill}, and \ref{proposition:grill}) and precisely describe what we mean by ``notions of size'' (Definitions~\ref{definition:collections-mesh}, \ref{definition:stack}, and \ref{definition:filter-grill}).
The main highlight, and organizing principle of this section, is the well-known ``mesh'' operator (Section~\ref{section:collections-mesh}) and its relation to notions of size.
Finally we end this section (Section~\ref{section:review-algebraic-structure}) with a brief review of the algebraic structure of the Stone--Čech compactification of a discrete semigroup.

In Section~\ref{section:derived-sets} we (re)introduce a more generalized concept of a ``derived'' set (see Remark~\ref{remark:derived-set} for other instances in the literature) along a collection (Definition~\ref{definition:derived-set}) and (re)prove some slight generalizations of elementary properties (Proposition~\ref{proposition:derived-set}) of derived sets.
Generalizing from the characterization of a product of two ultrafilters (see Section~\ref{section:review-algebraic-structure}), we then define a product of two collections (Definition~\ref{definition:product-of-collections}) and, as the main result of this section, prove some elementary properties of this product (Corollary~\ref{corollary:derived-set}).
In particular, this product is an associative binary operation that is isotone in each component.

Section~\ref{section:relative-notions} is the main and longest section of our article.
The two main results (Theorems~\ref{theorem:relative-syndetic-thick} and \ref{theorem:relative-piecewise-syndetic}) characterize relative notions of (piecewise) syndetic and thick sets in terms of derived sets.
Shuungula, Zelenyuk, and Zelenyuk \cite{Shuungula:2009ty} introduced the relative notion of (piecewise) syndetic sets which forms our starting point.
We also state a problem (Problem~\ref{problem:characterization-problem}) asking to characterize minimal and maximal ultrafilters with respect to a preorder defined in terms of derived sets. 
As an application of these characterizations, we define an algebraically interesting subset of \( \beta S \), which in some ways is analogous to the smallest ideal, (Theorem~\ref{theorem:relative-kernel} and Corollary~\ref{corollary:relative-kernel}). 

In Section~\ref{section:collectionwise} we define collectionwise relative piecewise syndetic sets (Definition~\ref{definition:collectionwise-relative-piecewise-syndetic}) and use this notion to characterize which collections are contained in ultrafilters of the algebraically interesting subset defined in Section~\ref{section:relative-notions} (Theorem~\ref{theorem:collectionwise-relative-piecewise-syndetic}).
Hindman and Lisan \cite{Hindman:1994aa} defined the notion of collectionwise piecewise syndetic to characterize which collections are contained in ultrafilters of the smallest ideal.
Recently, generalizing this result, Luperi Baglini, Patra, and Shaikh \cite{Luperi-Baglini:2023aa} defined a relative notion of collectionwise piecewise syndetic and characterized which collections are contained in ultrafilters of the smallest ideal of closed subsemigroups of \( \beta S \). 
Both definitions of collectionwise piecewise syndetic are complicated, while, in our view, our definition is simpler (because it uses the concept of a derived set along an ultrafilter). 
We end this section by proposing a definition for relative central sets (Definition~\ref{definition:relative-central}) using our relative notion of collectionwise piecewise syndetic, ask if this notion satisfies the Ramsey property (Question~\ref{question:partition-regular}), and provide a partial answer of ``yes'' under some additional assumptions (Corollary~\ref{corollary:relative-central-is-partition-regular}).

In a several of our proofs below, we follow a typical convention in mathematical practice: for three (or more) statments \( P \),  \( Q \),  and \( R \) we write ``\( P \iff Q \iff R \)'' to mean ``\( P \iff Q \) and \( Q \iff R \)'', and, similarly, we take ``\( P \implies Q \iff R \)'' to mean ``\( P \implies Q \) and \( Q \iff R \)''.

\section{Preliminaries: Notions of size and mesh}
\label{section:preliminaries}

We introduce terminology which, roughly speaking, captures different aspects of ``notion of size''.
Not all the terminology (such as `mesh', `stack', or `grill') is well known, but they have precedents in the literature (for instance, see \cite[Remarks~2.2 and 2.4]{Christopherson:2022wr}).
The mesh operator relates all the notions of size we consider.

All the results in this section go back at least as far as Choquet \cite{Choquet:1947aa}, and later in a more general framework of Galois connections between partially ordered sets in Schmidt \cite{Schmidt:1952aa, Schmidt:1953gm}.
Akins \cite[Ch.~2]{Akin1997} is a more contemporary reference, while Erné \cite{Erne:2004aa} and Erné, Koslowski, Melton, and Strecker \cite{Erne:1993aa} surveys the historical origins and applications of Galois connections.
The proof of the results in this section are routine, but for convenience, we provide the standard proofs.

\subsection{Collections and mesh}
\label{section:collections-mesh}

\begin{definition}
\label{definition:collections-mesh}
    Let \( S \) be a nonempty set.
    \begin{itemize}
    \item[(a)] 
        We call \( \mathcal{F} \subseteq \mathcal{P}(S) \) a \define{collection} on \( S \), which is \define{proper} if \( \emptyset \ne  \mathcal{F} \) and \( \emptyset \notin \mathcal{F} \). 
				
    \item[(b)]
    For collections \( \mathcal{F} \) the \define{mesh (operator)} is
    \[ 
        \mathcal{F}^* := \{ A \subseteq S : \text{for all } B \in \mathcal{F} \text { we have } A \cap B \ne \emptyset \}. 
    \]	
    \end{itemize}
\end{definition}

Given a collection \( \mathcal{F} \)  we consider \( A \subseteq S \) to be ``large'' if \( A \in \mathcal{F} \), and, in analogy with topological adherent points, \( A \) is ``nearly large'' if \( A \in \mathcal{F}^* \).
Intuitively, we usually also assume \emph{some} set is large and \emph{not} consider the empty set as large. 
Proper collections satisfy this assumption, but we also allow ``improper'' collections since it's sometimes convenient to consider them as well.

\begin{example}
\label{example:collection}
    Throughout this paper we will illustrate many of the concepts via the following examples.
    \begin{itemize}
        \item[(a)] 
            For instance, \( \bigl\{ \{k, k+1, \ldots \} : k \in \mathbb{N} \bigr\} \) is a proper collection on \( \mathbb{N} \) 
            with \( \bigl\{ \{k, k+1, \ldots \} : k \in \mathbb{N} \bigr\}^* = \{ A \subseteq \mathbb{N} : (\forall \, k \in \mathbb{N})(\exists \, x \in \mathbb{N}) \; k + x \in A \} \).

        \item[(b)] 
            Let \( \langle x_n \rangle_{n = 1}^\infty \) be a sequence in \( \mathbb{N} \) and put 
            \(
                \mathrm{FS}(\langle x_n \rangle_{n = 1}^\infty) := \bigl\{ \sum_{n \in H} x_n : H \in \mathcal{P}_f(\mathbb{N}) \bigr\}
            \).
            Then \( \{ \mathrm{FS}(\langle x_n \rangle_{n = k}^\infty) : k \in \mathbb{N} \} \) is also a proper collection on \( \mathbb{N} \)
            with \( \{ \mathrm{FS}(\langle x_n \rangle_{n = k}^\infty) : k \in \mathbb{N} \}^* = \{ A \subseteq \mathbb{N} : (\forall \, k \in \mathbb{N})(\exists \, H \in \mathcal{P}_f(\{k, k+1, \ldots \}) \; \sum_{n \in H} x_n \in A \} \).
    \end{itemize} 
\end{example}

Partially order the set of all collections on \( S \) (that is, \( \mathcal{P}\bigl(\mathcal{P}(S)\bigr) \)) by set inclusion \( \subseteq \).
The mesh is a Galois connection on \( (\mathcal{P}\bigl(\mathcal{P}(S)\bigr), \subseteq) \):
\begin{proposition}
\label{proposition:mesh-operator}
    Let \( S \) be a nonempty set.
    For collections \( \mathcal{F}_1 \) and \( \mathcal{F}_2 \) on \( S \) we have \( \mathcal{F}_1 \subseteq \mathcal{F}_2^* \) if and only if \( \mathcal{F}_2 \subseteq \mathcal{F}_1^* \).
\end{proposition}
\begin{proof}
    By symmetry (of the argument) it suffices to only show one direction, say the left-to-right direction (\( \Rightarrow \)).
    Let \( A \in \mathcal{F}_2 \).
    For \( B \in \mathcal{F}_1 \) by assumption we have \( B \in \mathcal{F}_2^* \) and so \( A \cap B \ne \emptyset \).
    Hence \( A \in \mathcal{F}_1^* \).
\end{proof}

Observe \( \emptyset = \mathcal{F} \iff \emptyset \in \mathcal{F}^* \).
(To see this, note \( \{\emptyset\}^* = \{ A \subseteq S : A \cap \emptyset \ne \emptyset \} = \emptyset \) and, by Proposition~\ref{proposition:mesh-operator}, we have \( \mathcal{F} \subseteq \{\emptyset\}^* = \emptyset \iff \{ \emptyset \} \subseteq \mathcal{F}^* \iff \emptyset \in \mathcal{F}^*  \).)
From essentially similar reasoning we immediately obtain a few properties of the mesh (some of which are equivalent to Proposition~\ref{proposition:mesh-operator}).
\begin{corollary}
\label{corollary:mesh-operator}
    Let \( S \) be a nonempty set. 
    \begin{itemize}
    \item[(a)]
        For every collection \( \mathcal{F} \) on \( S \) we have \( \mathcal{F} \subseteq \mathcal{F}^{**} \).
			
    \item[(b)]
	For all collections \( \mathcal{F}_1, \mathcal{F}_2 \) on \( S \) we have \( \mathcal{F}_1 \subseteq \mathcal{F}_2 \) implies \( \mathcal{F}_2^* \subseteq \mathcal{F}_1^* \) and \( \mathcal{F}_1^{**} \subseteq \mathcal{F}_2^{**} \).
	In particular, \( \emptyset \in \mathcal{F} \) if and only if \( \emptyset = \mathcal{F}^* \).
	Hence \( \mathcal{F} \) is proper if and only if \( \mathcal{F}^* \) is proper.
			
    \item[(c)]
	For \( \mathcal{F} \) a collection on \( S \) we have \( \mathcal{F}^* = \mathcal{F}^{***} \).
        In particular, for collections \( \mathcal{F}_1, \mathcal{F}_2 \) we have \( \mathcal{F}_1 \subseteq \mathcal{F}_2^{**} \) if and only if \( \mathcal{F}_1^{**} \subseteq \mathcal{F}_2^{**} \).
			
    \item[(d)]
	The set of fixed points of the function \( \mathcal{F} \mapsto \mathcal{F}^{**} \) is precisely the image of the function \( \mathcal{F} \mapsto \mathcal{F}^* \), that is,
    \[
        \{ \mathcal{F} : \mathcal{F} = \mathcal{F}^{**} \} = \{ \mathcal{F}^* :  \mathcal{F} \text{ collection on } S \}.
    \]
			
    \item[(e)]
        For \( \mathcal{F} \) a collection on \( S \) we have \( \mathcal{F}^* = \max \{ \mathcal{F}_1  : \mathcal{F} \subseteq \mathcal{F}_1^* \} \).
	In particular, \( \emptyset^* = \mathcal{P}(S) \) and \( \mathcal{P}(S)^* = \emptyset \).
			
    \item[(f)]
        Let \( \langle \mathcal{F}_i \rangle_{i \in I} \) be a family of collections on \( S \).
        Then \( \bigl(  \bigcup_{i \in I} \mathcal{F}_i \bigr)^* = \bigcap_{i \in I} \mathcal{F}_i^* \), where as usual if \( \langle \mathcal{F}_i \rangle_{i \in I} \) is the empty family we take \( \bigcap_{i \in I} \mathcal{F}_i = \mathcal{P}(S) \).		
    \end{itemize}	
\end{corollary}
\begin{proof}
	\textbf{(a):} 
	Trivially, \( \mathcal{F}^* \subseteq \mathcal{F}^* \), and so, by Proposition~\ref{proposition:mesh-operator}, we have \( \mathcal{F} \subseteq \mathcal{F}^{**} \).
	
	\smallskip
	
	\textbf{(b):}
	By (a) \( \mathcal{F}_1 \subseteq \mathcal{F}_2 \subseteq \mathcal{F}_2^{**} \), and so Proposition~\ref{proposition:mesh-operator} yields \( \mathcal{F}_2^* \subseteq \mathcal{F}_1^* \).
	Then \( \mathcal{F}_1^{**} \subseteq \mathcal{F}_2^{**} \) follows from \( \mathcal{F}_2^* \subseteq \mathcal{F}_1^* \).
	
	For the `in particular' statement if \( \emptyset \in \mathcal{F} \), then \( \{\emptyset\} \subseteq \mathcal{F} \) implies \( \mathcal{F}^* \subseteq \{\emptyset\}^* = \emptyset \) (where this last equality is from the justification of the observation above).
	Conversely, if \( \emptyset \notin \mathcal{F} \), then \( S \in \mathcal{F}^* \) and so \( \emptyset \ne \mathcal{F}^* \).
		
	\smallskip
	
	\textbf{(c):}
	By (a) it suffices to only show \( \mathcal{F}^{***} \subseteq \mathcal{F}^* \) (which follows from (a) and (b)).

	To see the `in particular' statement, we have the following string of biconditionals:
	\begin{align*}
		\mathcal{F}_1 \subseteq \mathcal{F}_2^{**} &\iff \mathcal{F}_2^* \subseteq \mathcal{F}_1^* &\text{(Proposition~\ref{proposition:mesh-operator})}\\
		&\iff \mathcal{F}_2^* \subseteq \mathcal{F}_1^* = \mathcal{F}_1^{***} \\
		&\iff \mathcal{F}_1^{**} \subseteq \mathcal{F}_2^{**}. &\text{(Proposition~\ref{proposition:mesh-operator})}
	\end{align*}
	
	\smallskip
	
	\textbf{(d):}
	The left-to-right inclusion (\( \subseteq \)) is trivial, while the right-to-left inclusion (\( \supseteq \)) follows from (c).
	
	\smallskip
	
	\textbf{(e):}
	Note (a) shows \( \{ \mathcal{F}_1 : \mathcal{F} \subseteq \mathcal{F}_1^* \} \)  contains \( \mathcal{F}^{*} \).
	If \( \mathcal{F}_1 \) is a collection on \( S \) with \( \mathcal{F} \subseteq \mathcal{F}_1^* \), then apply Proposition~\ref{proposition:mesh-operator} to obtain \( \mathcal{F}_1 \subseteq \mathcal{F}^* \).
	
	For the `in particular' statement, we have \( \emptyset^* = \max \{ \mathcal{F} : \emptyset\subseteq \mathcal{F}^* \} = \max \mathcal{P}\bigl(\mathcal{P}(S)\bigr) = \mathcal{P}(S) \).
	To then see \( \mathcal{P}(S)^* = \emptyset \) note from the observation above
	\( \{ \mathcal{F} : \mathcal{P}(S) \subseteq \mathcal{F}^* \} \subseteq \{ \mathcal{F} : \emptyset \in \mathcal{F}^* \} = \{ \emptyset \} \), and so, it follows \( \mathcal{P}(S)^* = \max \{ \emptyset \} = \emptyset \).
	
	\smallskip
	
	\textbf{(f):} 
    It suffices to show \( \mathcal{F} \subseteq \bigl( \bigcup_{i \in I} \mathcal{F}_i \bigr)^* \) if and only if \( \mathcal{F} \subseteq \bigcap_{i \in I} \mathcal{F}_i^*\). 
    To that end note we have the following string of biconditionals:
    \begin{align*}
        \mathcal{F} \subseteq \bigl( \bigcup_{i \in I} \mathcal{F}_i \bigr)^* &\iff \bigcup_{i \in I} \mathcal{F}_i \subseteq \mathcal{F}^* &\text{(Proposition~\ref{proposition:mesh-operator})}\\
        &\iff \text{for all } i \in I \text{, } \mathcal{F}_i \subseteq \mathcal{F}^* \\
        &\iff \text{for all } i \in I \text{, } \mathcal{F} \subseteq \mathcal{F}_i^* &\text{(Proposition~\ref{proposition:mesh-operator})} \\
        &\iff \mathcal{F} \subseteq \bigcap_{i \in I} \mathcal{F}_i^*.
    \end{align*}

\end{proof}

\begin{remark}
\label{remark:closure-operator}
	The `in particular' statement in Corollary~\ref{corollary:mesh-operator}(c) is equivalent to asserting \( \mathcal{F} \mapsto \mathcal{F}^{**} \)  is a closure operator on the set of all collections (for instance, see \cite[Proposition 3(4)]{Erne:1993aa}).
	Hence \( \mathcal{F}^{**} \) is the smallest (see Proposition~\ref{proposition:stack}(c)) upward closed collection containing \( \mathcal{F} \).
\end{remark}

\subsection{Stacks: upward closed collections}
\label{section:stack}

\begin{definition}
\label{definition:stack}
	Let \( S \) be a nonempty set and let \( \mathcal{F} \) be a collection on \( S \).
	We call \( \mathcal{F} \) a \define{stack} on \( S \) if and only if \( \mathcal{F} \) is upward closed: \( A \in \mathcal{F} \) and \( A \subseteq B \subseteq S \) implies \( B \in \mathcal{F} \).
\end{definition}

Intuitively, a stack also captures another property we may want to assume for collections of ``large'' sets~---~any set containing a large set is large too.
Most of the collections we consider will be proper stacks.

Neither of the collections in Example~\ref{example:collection} are stacks but they both easily generate proper stacks:
\begin{example}
\label{example:stacks}    
    \mbox{}
    \begin{itemize}
        \item[(a)]
            Put \( \mathcal{C} := \bigl\{ A \subseteq \mathbb{N} : (\exists \, k \in \mathbb{N}) \; \{k, k+1, \ldots, \} \subseteq A \bigr\} \).
            Then \( \mathcal{C} \) is a proper stack on \( \mathbb{N} \) with
            \( \mathcal{C}^* = \bigl\{ A \subseteq \mathbb{N} : (\forall \, k \in \mathbb{N}) (\exists \, x \in \mathbb{N}) \; k + x \in A \} \)   
            (compare with Example~\ref{example:collection}(a)).

        \item[(b)]
            Put \( \mathcal{H} :=\{ A \subseteq \mathbb{N} :  (\exists\, k \in \mathbb{N}) \; \mathrm{FS}(\langle x_n \rangle_{n=k}^\infty) \subseteq A \} \).
            Then \( \mathcal{H} \) is a proper stack on \( \mathbb{N} \) with
            \( \mathcal{H}^* = \{ A \subseteq \mathbb{N} : (\forall \, k \in \mathbb{N})(\exists \, H \in \mathcal{P}_f(\{k, k+1, \ldots \}) \; \sum_{n \in H} x_n \in A \} \) 
            (again, compare with Example~\ref{example:collection}(b)).
            
    \end{itemize}
\end{example}

\todo[inline, color=gray!10]{\textsf{Idea}(mainly self note for John): The filter \( \mathcal{H} \) is analogous to \href{https://en.wikipedia.org/wiki/Filters_in_topology#Filters_and_nets}{``elementary'' (eventuality) filters} in topology. 
Any useful applications/point-of-view regarding this filter, especially related to considering filters as a lattice. 
Mainly, I'm thinking about the use of the pseudocomplement operator.}

The next result shows stacks are precisely the fixed points of the function \( \mathcal{F} \mapsto \mathcal{F}^{**} \):
\begin{proposition}
\label{proposition:stack}
	Let \( S \) be a nonempty set and let \( \mathcal{F} \) be a collection on \( S \).
	\begin{itemize}
		\item[(a)] 
			\( \mathcal{F}^* \) and \( \mathcal{F}^{**} \) are both stacks on \( S \), which are proper if and only if \( \mathcal{F} \) is proper.

		\item[(b)]
			If \( \mathcal{F} \) is a stack, then for all \( A \subseteq S \) we have \( A \in \mathcal{F}^* \) if and only if \( S \setminus A \notin \mathcal{F} \).
			
		\item[(c)]
			\( \mathcal{F}^{**} = \{ A \subseteq S : \text{there exists } B \in \mathcal{F} \text{ such that } B \subseteq A \} \).
			
		\item[(d)]
			\( \mathcal{F} \) is a stack on \( S \) if and only if \( \mathcal{F} = \mathcal{F}^{**} \).
			
		\item[(e)]
			Let \( \langle \mathcal{F}_i \rangle_{i \in I} \) be a family of stacks.
			Then \( \bigl(  \bigcap_{i \in I} \mathcal{F}_i \bigr)^* = \bigcup_{i \in I} \mathcal{F}_i^* \), where as usual if \( \langle \mathcal{F}_i \rangle_{i \in I} \) is the empty family we take \( \bigcap_{i \in I} \mathcal{F}_i = \mathcal{P}(S) \).		
	\end{itemize}	
\end{proposition}
\begin{proof}
	\textbf{(a):}
		To see \( \mathcal{F}^* \) is a stack let \( A \in \mathcal{F}^* \) and \( A \subseteq B \subseteq S \).
		For \( C \in \mathcal{F} \) we have \( \emptyset \ne A \cap C \subseteq B \cap C \).
		Hence \( B \in \mathcal{F}^* \).	
		It's then immediate \( \mathcal{F}^{**} \) is also a stack.
		The assertion on properness follows from Corollary~\ref{corollary:mesh-operator}(b).
	
	\smallskip	
	
	\textbf{(b):}
		If \( A \in \mathcal{F}^* \), then \( A \cap (S \setminus A) = \emptyset \) implies \( S \setminus A \notin \mathcal{F} \).
		Conversely, if \( A \notin \mathcal{F}^* \), then there exists \( B \in \mathcal{F} \) with \( A \cap B = \emptyset \), that is, \( B \subseteq S \setminus A \).
		Since \( \mathcal{F} \) is a stack we have \( S \setminus A \in \mathcal{F} \).

	\smallskip
	
	\textbf{(c):}
	We have the following string of biconditionals:
	\begin{align*}
		A \in \mathcal{F}^{**} &\iff S \setminus A \notin \mathcal{F}^* &\text{(by (a) and (b))} \\
		&\iff \text{there is } B \in \mathcal{F} \text{ such that } B \cap (S \setminus A) = \emptyset \\
		&\iff \text{there is } B \in \mathcal{F} \text{ such that } B \subseteq A.
	\end{align*}
	
	\smallskip
		
	\textbf{(d):}
	For the forward direction, by (b) we have \( A \in \mathcal{F}^{**} \iff S \setminus A \notin \mathcal{F}^* \iff A \in \mathcal{F} \).
	The converse direction follows from (a).
	
	\smallskip
	
	\textbf{(e):}
	We have 
	\begin{align*}
		\Bigl(  \bigcap_{i \in I} \mathcal{F}_i \Bigr)^* &= \Bigl(  \bigcap_{i \in I} \mathcal{F}_i^{**} \Bigr)^* &\text{(by (d))} \\
		&= \Bigl( \bigcup_{i \in I} \mathcal{F}_i^* \Bigr)^{**} &\text{(Corollary~\ref{corollary:mesh-operator}(f))} \\
		&= \bigcup_{i \in I} \mathcal{F}_i^*. &\text{(union is a stack and by (d))}
	\end{align*}
\end{proof}

\subsection{(Ultra)filters and grills}
\label{section:(ultra)filters-grills}

We now state the last three notions of size and observe how they are related via the mesh operator.
\begin{definition}
\label{definition:filter-grill}
	Let \( S \) be a nonempty set and let \( \mathcal{F} \) be a stack on \( S \).
	\begin{itemize}
		\item[(a)]
			\( \mathcal{F} \) is a \define{filter} on \( S \) if and only if \( \emptyset \ne \mathcal{F} \) and \( \mathcal{F} \) is closed under finite intersections: \( A_1 \in \mathcal{F} \) and \( A_2 \in \mathcal{F} \)  implies \( A_1 \cap A_2 \in \mathcal{F} \).
			
		\item[(b)]
			\( \mathcal{F} \) is a \define{grill} on \( S \) if and only if \( \emptyset \notin \mathcal{F} \) and \( \mathcal{F} \) satisfies the Ramsey property: \( A_1 \cup A_2 \in \mathcal{F} \) implies either \( A_1 \in \mathcal{F} \) or \( A_2 \in \mathcal{F}_2 \).
			
		\item[(c)]
			An \define{ultrafilter} on \( S \) is both a filter and grill.
	\end{itemize}	
\end{definition}

Elements of a filter \( \mathcal{F} \) are qualitatively ``full measure'' sets, while elements of the corresponding grill \( \mathcal{F}^* \) (see Proposition~\ref{proposition:filter-grill}(a)) are qualitatively ``positive measure'' sets. 
(For instance, consider the filter \( \mathcal{F} \) consisting of measurable subsets of the interval \( [0, 1] \) with Lebesgue measure 1, then \( \mathcal{F}^* \) are precisely measurable subsets of \( [0, 1] \) with positive Lebesgue measure.)

\begin{example}
\label{example:filters}
    \mbox{}
    \begin{itemize}
        \item[(a)]
            Both the stacks \( \mathcal{C} \) and \( \mathcal{H} \) in Example~\ref{example:stacks} are actually filters on \( \mathbb{N} \).
            And, it's not hard to show (or see Proposition~\ref{proposition:filter-grill}(a)) that both \( \mathcal{C}^* \) and \( \mathcal{H}^* \) are grills on \( \mathbb{N} \) too.

        \item[(b)]
            A bit more generally, let \( S \) is a nonempty set and put  \( \mathcal{C} := \{ A \subseteq S : S \setminus A \text{ is finite}\} \).
            Then \( \mathcal{C} \) is the well known cofinite filter with \( \mathcal{C}^* = \{ A \subseteq S : A \text{ is infinite} \} \) a grill.
            Note \( \mathcal{C} \) is proper if and only if \( S \) is infinite.

        \item[(c)]
            Put \( \mathcal{D} := \{ A \subseteq \mathbb{N} : \mathsf{bd}^\star(\mathbb{N} \setminus A) = 0 \} \), where \( \mathsf{bd}^\star(A) \) is the upper Banach density of \( A \) (see Section~\ref{section:review-algebraic-structure}). 
            Then \( \mathcal{D}^* = \{ A \subseteq \mathbb{N} : \mathsf{bd}^\star(A) > 0 \} \),  and because \( \mathsf{bd}^\star \) is finitely subadditive (\( \mathsf{bd}^\star(A \cup B) \le \mathsf{bd}^\star(A) + \mathsf{bd}^\star(B) \)), we have \( \mathcal{D}^* \) is a proper grill.
            From this it is not hard to show (or again, see Proposition~\ref{proposition:filter-grill}(a)) \( \mathcal{D} \) is a proper filter.   
    \end{itemize}
\end{example}

We allow \( \mathcal{P}(S) \) as the improper filter and \( \emptyset \) as the improper grill, but it follows that ultrafilters are proper.
As noted by Choquet \cite{Choquet:1947aa} and Schmidt \cite{Schmidt:1953gm}, the mesh gives an order anti-isomorphism between filters and grills where ultrafilters are precisely the fixed points of \( \mathcal{F} \mapsto  \mathcal{F}^* \) when restricted to filters:
\begin{proposition}
\label{proposition:filter-grill}
	Let \( S \) be a nonempty set and let \( \mathcal{F} \) be a stack.
	\begin{itemize}
		\item[(a)]
			\( \mathcal{F} \) is a filter if and only if \( \mathcal{F}^* \) is a grill.
			
		\item[(b)] 
			If \( \mathcal{F} \) is a filter, then \( \mathcal{F} \) is proper if and only if \( \mathcal{F} \subseteq \mathcal{F}^* \).
		
		\item[(c)]
			If \( \mathcal{F} \) is a filter, then \( \mathcal{F} \) is an ultrafilter if and only if \( \mathcal{F} = \mathcal{F}^* \).
			
		\item[(d)]
			If \( \mathcal{F} \) is a filter, then \( \mathcal{F} \) is an ultrafilter if and only if for all \( A \subseteq S \) we have \( A \in \mathcal{F} \iff S \setminus A \notin \mathcal{F} \).
			
		\item[(e)]
			Ultrafilters are precisely maximal proper filters.
	\end{itemize}	
\end{proposition}
\begin{proof}
	\textbf{(a):} 
	Corollary~\ref{corollary:mesh-operator}(b) asserts \( \mathcal{F} \) is proper if and only if \( \mathcal{F}^* \) is proper.	
	If \( \mathcal{F} \) is proper, then for instance see \cite[Proposition~2.5(e)]{Christopherson:2022wr} (or it can easily be verified directed).
	If \( \mathcal{F} \) is improper, then this follows from Corollary~\ref{corollary:mesh-operator}(e).
	
	\smallskip
	
	\textbf{(b):}
	By Proposition~\ref{proposition:stack}(a) and (b) we have \( \mathcal{F}^* = \{ A \subseteq S : S \setminus A \notin \mathcal{F} \} \).
	If \( \mathcal{F} \) is proper, then for \( A \in \mathcal{F} \) we have \( S \setminus A \notin \mathcal{F} \) (since \( \mathcal{F} \) is a filter), and so \( A \in \mathcal{F}^* \).
	Conversely, if \( \mathcal{F} \subseteq \mathcal{F}^* \), then since \( S \in \mathcal{F}^* \) ( a filter is a nonempty stack) we have \( \emptyset \not\in \mathcal{F} \) and  \( \emptyset \ne \mathcal{F} \) (since \( \mathcal{F} \) is a filter). 
	Therefore \( \mathcal{F} \) is proper.

	\smallskip
	
	\textbf{(c):}
	For instance see \cite[Proposition~2.5(f)]{Christopherson:2022wr}.
	
	\smallskip
	
	\textbf{(d):}
	By Proposition~\ref{proposition:stack}(b), note the condition `for all \( A \subseteq S \) we have \( A \in \mathcal{F} \iff S \setminus A \notin \mathcal{F} \)' is equivalent to `\( \mathcal{F} = \mathcal{F}^* \)'.
	Hence (d) follows from (c).
	
	\smallskip
	
	\textbf{(e):}
	Let \( p \) be an ultrafilter (see Section~\ref{section:review-algebraic-structure} for our notational choices for ultrafilters) and let \( \mathcal{F} \) be a filter with \( p \subseteq \mathcal{F} \).
	It suffices to show either \( p = \mathcal{F} \) or \( \mathcal{F} = \mathcal{P}(S) \).
	We may assume \( p \ne \mathcal{F} \) and pick \( A \in \mathcal{F} \setminus p \).
	By (d) we then have \( S \setminus A \in p \).
	Hence \( A \) and \( S \setminus A \) are both in \( \mathcal{F} \) and it follows that \( \mathcal{F} = \mathcal{P}(S) \).

    Now let \( p \) be a maximal proper filter.
    By (b), it suffices to show \( p^* \subseteq p \).
    Let \( A \in p^* \) and note the collection \( \{A \cap B : B \in p \} \) satisfies the finite intersection property.
    So it follows \( \{A \cap B : B \in p \}^{**} \) is a proper filter that contains \( p \) and has \( A \) as a member. 
    Therefore by maximality of \( p \) we have \( A \in p \).
\end{proof}

Next we define a binary operation on the set of all collections on \( S \), adopting Glasscock and Le \cite[Sec.~2.1.1]{Glasscock:2024aa} notation, that will allow us given any stack to produce a grill minimal with respect to a certain condition (Proposition~\ref{proposition:grill}(b)).
This and the following results essentially goes back to Choquet \cite{Choquet:1947aa}, but the elegant presentation via the binary operation goes back at least as early as Akin \cite[Ch.~2]{Akin1997}:
\begin{definition}
\label{definition:binary-operation}
	Let \( \mathcal{F}_1, \mathcal{F}_2 \) both be collections on \( S \) and define
	\[
		\mathcal{F}_1 \sqcap \mathcal{F}_2 := \{ A_1 \cap A_2 : A_1 \in \mathcal{F}_1 \text{ and } A_2 \in \mathcal{F}_2 \}.
	\]
\end{definition}
\begin{remark}
\label{remark:binary-operation}
	Akin denotes this by \( \mathcal{F}_1 \cdot \mathcal{F}_2 \), but we don't adopt this notation since we'll give another meaning to \( \mathcal{F}_1 \cdot \mathcal{F}_2 \) in Section~\ref{section:derived-sets}.
\end{remark}
This binary operation is clearly associative, commutative, has \( \emptyset \) as the zero element, has \( \{S\} \) as the identity, is isotone (\( \mathcal{F} \subseteq \mathcal{F}_1' \implies \mathcal{F} \sqcap \mathcal{F}_1 \subseteq \mathcal{F} \sqcap \mathcal{F}_1' \)), and a stack \( \mathcal{F} \) is a filter if and only if \( \mathcal{F} \sqcap \mathcal{F} \subseteq \mathcal{F} \).
Moreover, it also satisfies a Galois connections among all stacks on \( S \):

\begin{proposition}
\label{proposition:binary-operation}
    Let \( S \) be a nonempty set, let \( \mathcal{F} \) be a collection, and let \( \mathcal{F}_1, \mathcal{F}_2 \) both be stacks on \( S \).
    Then \( \mathcal{F} \sqcap \mathcal{F}_1 \subseteq \mathcal{F}_2 \) if and only if \( \mathcal{F}_1 \subseteq \bigl( \mathcal{F} \sqcap (\mathcal{F}_2)^* \bigr)^* \).    
\end{proposition}
\begin{proof}
    By Proposition~\ref{proposition:mesh-operator} we can, equivalently, show  \( \mathcal{F} \sqcap \mathcal{F}_1 \subseteq \mathcal{F}_2 \) if and only if \( \mathcal{F} \sqcap \mathcal{F}_2^* \subseteq \mathcal{F}_1^* \).

    \smallskip
    
    \textbf{(\( \Rightarrow \)):}
    Let \( A \in \mathcal{F} \) and \( A_2 \in \mathcal{F}_2^* \).
    For \( A_1 \in \mathcal{F}_1 \) we have, by assumption, \( A \cap A_1 \in \mathcal{F}_2 \) and so \( A \cap A_1 \cap A_2 \ne \emptyset \).
    Therefore \( A \cap A_2 \in \mathcal{F}_1^* \).

    \smallskip

    \textbf{(\( \Leftarrow \)):}
    Let \( A \in \mathcal{F} \) and \( A_1 \in \mathcal{F}_1 \) be such that \( A \cap A_1 \not\in \mathcal{F}_2 \). 
    Since \( \mathcal{F}_2 \) is a stack, by Proposition~\ref{proposition:stack}(b), we have
    \(
        (S \setminus A) \cup (S \setminus A_1) = S \setminus (A \cap A_1) \in \mathcal{F}_2^*
    \).
    Therefore
    \(
        A \setminus A_1 = A \cap \bigl( S \setminus (A \cap A_1) \bigr) \in \mathcal{F} \sqcap \mathcal{F}_2^*
    \)
    and \( (A \setminus A_1) \cap A_1 = \emptyset \) implies \( A \setminus A_1 \not\in \mathcal{F}_1^* \).
\end{proof}

\begin{proposition}
\label{proposition:grill}
    Let \( S \) be a nonempty set and let \( \mathcal{F}, \mathcal{F}_1, \mathcal{F}_2 \) all be stacks on \( S \).
   \begin{itemize}
         \item[(a)]
            \begin{itemize}
                \item[(i)] \( \mathcal{F}_1 \sqcap \mathcal{F}_2 \) is a stack on \( S \).

                \item[(ii)] \( \emptyset \ne \mathcal{F}_1 \sqcap \mathcal{F}_2 \) if and only if \( \emptyset \ne \mathcal{F}_1 \) and \( \emptyset \ne \mathcal{F}_2 \).
                In particular, \(  \emptyset \ne \mathcal{F}_1 \) and \( \emptyset \ne \mathcal{F}_2 \) implies \( \mathcal{F}_1 \cup \mathcal{F}_2 \subseteq \mathcal{F}_1 \sqcap \mathcal{F}_2 \).

                \item[(iii)] \( \emptyset \not\in \mathcal{F}_1 \sqcap \mathcal{F}_2 \) if and only if \( \mathcal{F}_1 \subseteq \mathcal{F}_2^* \) and \( \emptyset \not\in \mathcal{F}_1 \) and \( \emptyset \not\in \mathcal{F}_2 \).
            \end{itemize}

        \item[(b)]
            \( \mathcal{F} \sqcap (\mathcal{F}^*) \) is a grill on \( S \), which is proper if and only if \( \mathcal{F} \) is proper.
            In particular, \( \mathcal{F} \sqcap \mathcal{F}_1 \subseteq \mathcal{F} \) if and only if \( \mathcal{F} \sqcap (\mathcal{F}^*) \subseteq \mathcal{F}_1^* \).

     \end{itemize}
\end{proposition}
\begin{proof}
    \textbf{(a):}
    To see statement (i) let \( A \in \mathcal{F}_1 \sqcap \mathcal{F}_2 \) and \( A \subseteq B \subseteq S \).
    Pick \( A_1 \in \mathcal{F}_1 \) and \( A_2 \in \mathcal{F}_2 \) with \( A = A_1 \cap A_2 \).
    Then each \( A_i \cup B \in \mathcal{F}_i \), and so \( B = A \cup B = (A_1 \cap A_2 ) \cup B = (A_1 \cup B) \cap (A_2 \cup B) \in \mathcal{F}_1 \sqcap \mathcal{F}_2 \).

    Statements (ii) and (iii) follow immediately from the definitions.
    The `in particular' assertion in (ii), because \( \mathcal{F}_i \) is a stack, follows from \( \emptyset \ne \mathcal{F}_i \) if and only if \( S \in \mathcal{F}_i \).

    \smallskip

    \textbf{(b):}
    See, for instance, the proof of \cite[Proposition~2.5(h)]{Christopherson:2022wr} showing \( \mathcal{F} \sqcap (\mathcal{F}^*) \) is a grill on \( S \).
	(Note in \cite{Christopherson:2022wr} all stacks are proper, but besides the first sentence, the proof doesn't use properness.)

    By (a, ii) and (a, iii) we have \( \mathcal{F} \sqcap (\mathcal{F}^*) \) is proper if and only if \( \mathcal{F} \) is proper.
    The `in particular' assertion follows from Proposition~\ref{proposition:binary-operation}.
\end{proof}

\begin{example}
\label{example:grill}
    Put \( \mathcal{F} := \{ A \subseteq \mathbb{N} : 2\mathbb{N} +1 \subseteq A \text{ or } 2\mathbb{N}+2 \subseteq A \} \) and note \( \mathcal{F} \) is a proper stack and \( \mathcal{F}^* = \{ A \subseteq \mathbb{N} : (2\mathbb{N}+1) \cap A \ne \emptyset \text{ and } (2\mathbb{N}+2) \cap A \ne \emptyset \} \).
    Further \( \mathcal{F} \) is not a filter and, hence, \( \mathcal{F}^* \) is not a grill. 
    In this case, we have \( \mathcal{F} \sqcap (\mathcal{F}^*) = \{ A \subseteq \mathbb{N} : (2\mathbb{N} + 1) \cap A \ne \emptyset \text{ or } (2\mathbb{N} + 2) \cap A \ne \emptyset \} \) and   \( \left(\mathcal{F} \sqcap (\mathcal{F}^*)\right)^* = \{ A \subseteq \mathbb{N} : 2\mathbb{N} + 1 \subseteq A \text{ and } 2\mathbb{N}+2 \subseteq A \} \).
\end{example}

\subsection{Brief review of algebraic structure of the Stone--Čech compactification}
\label{section:review-algebraic-structure}
We'll end this section with a brief review of a bit more notation and the algebraic structure of the Stone--Čech compactification of a discrete semigroup.

\paragraph{Notation} 
We typically let \( S \) denote an infinite discrete semigroup with associative binary operation \( \cdot \), which we indicate by juxtaposition or by \( + \), if \( S \) is commutative.
For \( A \subseteq S \) and \( h \in S \) define \( h^{-1}A := \{ x \in S: hx \in A \} \), or \( -h + A := \{ x \in S : h + x \in A \} \) if \( S \) is commutative.
(This is just the preimage of the function \( x \mapsto hx \).)
Given a set \( X \) we let \( \mathcal{P}_f(X) \) denote the collection of all nonempty finite subsets of \( X \).
In \( (\mathbb{N}, +) \), where we give many of our examples, for \( k, \ell \in \mathbb{N} \) we put \( [k] := \{1, 2, \ldots, k\} \) and \( [k, \ell] := \{ k, k+1, \ldots, \ell \} \).

Let \( \beta S \) be the set of all ultrafilters on \( S \), and, following the typographical conventions of \cite{Hindman:2012tq} we typically let lowercase letters \( p, q, r \) denote ultrafilters on \( S \).
For \( A \subseteq S \) and \( \mathcal{F} \) a filter on \( S \) put \( \overline{A} := \{ p \in \beta S : A \in p \} \) and \( \overline{\mathcal{F}} := \{ p \in \beta S : \mathcal{F} \subseteq p \} \).
For proper filters \( \mathcal{F} \) we have \( p \in \overline{\mathcal{F}} \) if and only if \( p \subseteq \mathcal{F}^* \).

In Sections~\ref{section:relative-piecewise-syndetic} and \ref{section:collectionwise}, we let \( \bigtimes_{i \in I} A_i \) denote the Cartesian product of a family \( \langle A_i \rangle_{i \in I} \).

\medskip

We identify elements of \( S \) with the principal ultrafilters on \( \beta S \), and, as usual, pretend \( S \subseteq \beta S \).
The collection \( \{ \overline{A} : A \subseteq S \} \) forms a clopen basis for a compact Hausdorff topology on \( \beta S \), which is the Stone--Čech compactification on \( S \), where we have the closure \( c\ell_{\beta S}(A) = \overline{A} \).
There is a one-to-one correspondence between filters on \( S \) and closed subsets of \( \beta S \) via \( \mathcal{F} \mapsto \overline{\mathcal{F}} \) \cite[Theorem 3.20]{Hindman:2012tq}.

By \cite[Theorem~4.1]{Hindman:2012tq} the semigroup operation on \( S \) can be extended to a semigroup operation on \( \beta S \), and moreover for \( p, q \in \beta S \) we have \( A \in p \cdot q \iff \{ h \in S : h^{-1}A \in q \} \in p \) \cite[Theorem~4.12]{Hindman:2012tq}.
For \( \mathcal{F}, \mathcal{G} \) filters on \( S \) and \( q \in \beta S \) we put \( \overline{\mathcal{F}} \cdot \overline{\mathcal{G}} := \{ p \cdot q : p \in \overline{\mathcal{F}} \text{ and } q \in \overline{\mathcal{G}} \} \) and \( \overline{\mathcal{F}} \cdot q := \{p \cdot q : p \in \overline{\mathcal{F}} \} \).
Moreover, \( (\beta S, \cdot) \) is a compact \define{right topological semigroup}, which means for all \( q \in \beta S \) we have \( p \mapsto p \cdot q \) is continuous for all \( p \in \beta S \).
\emph{Any} compact right topological Hausdorff semigroup has idempotent elements \cite[Theorem~2.5]{Hindman:2012tq}, and a smallest ideal which is the union of all minimal (with respect to inclusion) left ideals and also contains idempotents itself, which are called \define{minimal idempotents} \cite[Theorem~2.8]{Hindman:2012tq}.
For a semigroup \( T \) we let \( K(T) \) denote the smallest ideal of \( T \), and let \( E(T) \) denote the collection of all idempotent elements in \( T \).

While all of our results are stated in terms of arbitrary semigroups \( (S, \cdot) \), all of our examples are given in the positive integers under addition \( (\mathbb{N}, +) \).
In particular, the \define{upper Banach density} for \( A \subseteq \mathbb{N} \) is 
\[
    \mathsf{bd}^\star(A) := \sup \{ \alpha \in [0, 1] : (\forall \, n \in \mathbb{N}) \; \{ x \in \mathbb{N} : |A \cap ([n] + x)| \ge \alpha n \} \text{ is infinite}  \}.
\]
Hence \( \mathsf{bd}^\star(A) > \alpha \) means there are arbitrarily long blocks of consecutive positive integers where the proportion of elements of \( A \) contained in infinitely many blocks of a given length is at least \( \alpha \). 
As is well known this form of the upper Banach density is equivalent to several other possible definitions (for instance see \cite{Grekos2010}).

\begin{remark}
\label{remark:left-right-and-improper}
	\mbox{}
	\begin{itemize}
		\item[(a)]
			Several of our references take \( \beta S \) to be left topological instead of right topological. 
			In cases were we cite a particular result (or definition), the appropriate left-right switches of their proofs and statements give the appropriate right topological version.
			
		\item[(b)]
			Unlike most of our references we allow \( \mathcal{P}(S) \) and \( \emptyset \) as the improper filter and grill, respectively.
			When citing a previous result we either include the assumption of properness in the statement or simply note in the proof why the improper case is also valid (often for either vacuous or trivial reasons).
	\end{itemize}	
\end{remark}

\section{Derived sets and products of collections}
\label{section:derived-sets}

This section's goal is to (re)introduce the concept of a ``derived'' set along a collection on a semigroup \( S \) (Definition~\ref{definition:derived-set}), define another associative binary operation on collections (Definition~\ref{definition:product-of-collections}), and prove a few fundamental properties on derived sets and this binary operation (Proposition~\ref{proposition:derived-set} and Corollary~\ref{corollary:derived-set}).
The results in this section are (slight) generalizations of previous observations throughout the literature on algebra in the Stone--Čech compactification. 
In particular, we note Glasscock and Le \cite[Sec.~2]{Glasscock:2024aa} obtained similar results for stacks on the positive integers under addition \( (\mathbb{N}, +) \), but their proofs and statements readily generalize to arbitrary semigroups.
Similar to Section~\ref{section:preliminaries}, the proofs of this section are also routine, but for convenience we provide the full details. 

\begin{definition}
\label{definition:derived-set}
	Let \( S \) be a semigroup, let \( \mathcal{F} \) be a collection on \( S \), and let \( A \subseteq S \).
	Put 
	\[ 
		A'(\mathcal{F}) := \{ h \in S : h^{-1}A \in \mathcal{F} \}, 
	\]
	which we call the \define{derived} set of \( A \) along \( \mathcal{F} \).
\end{definition}
\begin{remark}
\label{remark:derived-set}
	The notation in Definition~\ref{definition:derived-set} is borrowed from \cite[Remark~2.18]{Griffin:2024aa}, but other notations often appear.
	For instance in Papazyan \cite{Papazyan1989, Papazyan:1990aa} as \( \Omega_\mathcal{F}(A) \), in Beiglböck \cite{Beiglbock:2011aa} as \( q^{-1}A \) (where \( q \) is an ultrafilter), and recently in  Glasscock and Le \cite{Glasscock:2024aa} as \( A - \mathcal{F} \) for derived sets in \( (\mathbb{N}, +) \).
	Also see Tserunyan \cite{Tserunyan:2016aa} for a related (but stronger notion) of differentiability of subsets of a semigroup.
\end{remark}

\begin{example}
\label{example:derived-set}
    Recall the filters \( \mathcal{C} \), \( \mathcal{H} \), and \( \mathcal{D} \) from Example~\ref{example:filters} and let \( A \subseteq \mathbb{N} \).
    In \( (\mathbb{N}, +) \) we have the derived set of \( A \) along several collections as follows.
    \begin{itemize}
        \item[(a)] 
            \( A'(\mathcal{C}) = \bigl\{ h \in \mathbb{N} : (\exists \, k \in \mathbb{N}) \; h + \{k, k+1, \ldots\} \subseteq A \bigr\} \) and
            \( A'(\mathcal{C}^*) = \bigl\{ h \in \mathbb{N} : (\forall \, k \in \mathbb{N})(\exists \, x \in \mathbb{N}) \; h + k + x \in A \bigr\} \).
            In \( (\mathbb{N}, +) \) the cofinite filter \( \mathcal{C} \) is translation-invariant, that is, for every \( h \in \mathbb{N} \) we have \( A \in \mathcal{C} \) if and only if \( -h+A \in  \mathcal{C} \).
            This means \( A'(\mathcal{C}) \ne \emptyset \iff A'(\mathcal{C}) = \mathbb{N} \iff A \in \mathcal{C} \) and \( A'(\mathcal{C}^*) \ne \emptyset \iff A'(\mathcal{C}^*) = \mathbb{N} \iff A \in \mathcal{C}^* \) .

        \item[(b)] 
            \( A'(\mathcal{H}) = \{ h \in \mathbb{N} : (\exists \, k \in \mathbb{N}) \; h +  \mathrm{FS}(\langle x_n \rangle_{n = k}^\infty ) \subseteq A \} \) and 
            \( A'(\mathcal{H}^*) \{ h \in \mathbb{N} : (\forall \, k \in \mathbb{N})(\exists \, H \in \mathcal{P}_f(\{k, k+1, \ldots\}) \; h + \sum_{n \in H} x_n \in A \} \).

        \item[(c)]
            Similarly to (a), upper Banach density is translation-invariant, that is, for every \( h \in \mathbb{N} \) we have \( \mathsf{bd}^\star(-h+A) = \mathsf{bd}^\star(A) \).
            Hence it follows that \( A'(\mathcal{D}) = \{ h \in \mathbb{N} : \mathsf{bd}^\star(\mathbb{N} \setminus A) = 0 \} \)
            and
            \( A'(\mathcal{D}^*) = \{ h \in \mathbb{N} : \mathsf{bd}^\star(A) > 0 \} \).
            Therefore \( A'(\mathcal{D}) \ne \emptyset \iff A'(\mathcal{D}) = \mathbb{N} \iff A \in \mathcal{D} \) and  \( A'(\mathcal{D}^*) \ne \emptyset \iff A'(\mathcal{D}^*) = \mathbb{N} \iff A \in \mathcal{D}^* \).

    \end{itemize}
    
\end{example}

Observe that \( A'(\emptyset) = \emptyset \) and \( A'(\mathcal{P}(S)) = S \).
As is known, we can characterize syndetic and thick sets in terms of derived sets: \( A \) is syndetic if and only if for all \( q \in \beta S \) we have \( A'(q)  \ne \emptyset \); \( A \) is thick if and only if there exists \( q \in \beta S \) such that \( A'(q) = S \); while  \( A \) is piecewise syndetic if and only if there exists \( q \in \beta S \) such that \( A'(q) \) is syndetic.
(We won't stop to verify these facts here since they will follow from our more general   Corollary~\ref{corollary:relative-syndetic-thick} and Theorem~\ref{theorem:relative-piecewise-syndetic}.)

Next we record a few properties that immediately follow from Definition~\ref{definition:derived-set}.
Proposition~\ref{proposition:derived-set} is a generalization of Papazyan's observation \cite[Lemma~1.3]{Papazyan:1990aa}. %
\begin{proposition}
\label{proposition:derived-set}
	Let \( S \) be a semigroup.
	\begin{itemize}
		\item[(a)]
		Let \( \mathcal{F}, \mathcal{G} \) both be stacks on \( S \) and let \( A_1,  A_2 \subseteq S \)
		\begin{itemize}
			\item[(i)]
			 \( (A_1 \cap A_2)'(\mathcal{F}) \subseteq A_1'(\mathcal{F}) \cap A_2'(\mathcal{F}) \).
			 If \( \mathcal{F} \) is a filter, then equality holds.
			 
			\item[(ii)]
			\( A_1'(\mathcal{G}) \cup A_2'(\mathcal{G}) \subseteq (A_1 \cup A_2)'(\mathcal{G}) \).
			If \( \mathcal{G} \) is a grill, then equality holds.
			
			\item[(iii)]
			If \( q \) is an ultrafilter on \( S \), then \( (A_1 \cap A_2)'(q) = A_1'(q) \cap A_2'(q) \) and \( A_1'(q) \cup A_2'(q) = (A_1 \cup A_2)'(q) \).
			
			\item[(iv)]
			\( A_1'(\mathcal{F}) \cap A_2'(\mathcal{G}) \subseteq (A_1 \cap A_2)'(\mathcal{F} \sqcap \mathcal{G}) \).
		\end{itemize}
		
		\item[(b)]
			Let \( \mathcal{F}, \mathcal{G} \) both be stacks on \( S \) and let \( A \subseteq S \).
			\begin{itemize}
				\item[(i)]
				\( ( S \setminus A)'(\mathcal{F}) = S \setminus A'(\mathcal{F}^*) \).
				
				\item[(ii)]
				If \( \mathcal{F} \) is a proper filter, then \( (S \setminus A)'(\mathcal{F}) \subseteq S \setminus A'(\mathcal{F}) \).
				
				\item[(iii)]
				If \( \mathcal{G} \) is a proper grill, then \( S \setminus A'(\mathcal{G}) \subseteq (S \setminus A)'(\mathcal{G}) \).
				
				\item[(iv)]
				If \( q \) is an ultrafilter, then \( S \setminus A'(q) = (S \setminus A)'(q) \).
				In particular, for \( B \subseteq S \) we have \( (A \setminus B)'(q) = A'(q) \setminus B'(q) \).
			\end{itemize}
			
		\item[(c)]
			\begin{itemize}
				\item[(i)]
				If \( \mathcal{F}_1, \mathcal{F}_2 \) are both collections on \( S \) and \( A \subseteq S \), then \( \mathcal{F}_1 \subseteq \mathcal{F}_2 \) implies \( A'(\mathcal{F}_1) \subseteq A'(\mathcal{F}_2) \).
					
				\item[(ii)]
				If \( \langle \mathcal{F}_i \rangle_{i \in I} \) is a family of collections on \( S \) and \( A \subseteq S \), then  \( A'(\bigcup_{i \in I} \mathcal{F}_i) = \bigcup_{i \in I} A'(\mathcal{F}_i) \) and \( A'(\bigcap_{i \in I} \mathcal{F}_i) = \bigcap_{i \in I} A'(\mathcal{F}_i) \).

				\item[(iii)]
				If \( \mathcal{F} \) is a stack and \( A_1 \subseteq A_2 \), then \( A_1'(\mathcal{F}) \subseteq A_2'(\mathcal{F}) \).
			\end{itemize}

		\item[(d)]
			If \( \mathcal{F} \) is a collection on \( S \), \( A \subseteq S \), and \( g \in S \), then \( (g^{-1}A)'(\mathcal{F}) = g^{-1}\bigl( A'(\mathcal{F}) \bigr)\).

	\end{itemize}
\end{proposition}
\begin{proof}
	\textbf{(a):}
	To see statement (i) note
	\begin{align*}
		h \in (A_1 \cap A_2)'(\mathcal{F}) &\iff h^{-1}A_1 \cap h^{-1}A_2 = h^{-1}(A_1 \cap A_2) \in \mathcal{F} \\
		&\implies h^{-1}A_1 \in \mathcal{F} \text{ and } h^{-1}A_2 \in \mathcal{F} &\text{(since \( \mathcal{F} \) is a stack)} \\
		&\iff h \in A_1'(\mathcal{F}) \text{ and } h \in A_2'(\mathcal{F}) \\
		&\iff h \in A_1'(\mathcal{F}) \cap A_2'(\mathcal{F}).
	\end{align*}
	When \( \mathcal{F} \) is a filter, observe the above forward implication `\( \implies \)' is also valid when reversed `\( \impliedby \)'.
	
	Similarly, to see statement (ii) note
	\begin{align*}
		h \in A_1'(\mathcal{G}) \cup A_2'(\mathcal{G}) &\iff h^{-1}A_1 \in \mathcal{G} \text{ or } h^{-1}A_2 \in \mathcal{G} \\
		&\implies h^{-1}(A_1 \cup A_2) = h^{-1}A_1 \cup h^{-1}A_2 \in \mathcal{G} &\text{(since \( \mathcal{G} \) is a stack)} \\
		&\iff h \in (A_1 \cup A_2)'(\mathcal{G}).
	\end{align*}
	When \( \mathcal{G} \) is a grill, observe the above forward implication `\( \implies \)' is also valid when reversed `\( \impliedby \)'. 
	
	Statement (iii) is then immediate from (i) and (ii).
	
	To see statement (iv) note 
	\begin{align*}
		h \in A_1'(\mathcal{F}) \cap A_2'(\mathcal{G}) &\iff h^{-1}A_1 \in \mathcal{F} \text{ and } h^{-1}A_2 \in \mathcal{G} \\ 
		&\implies h^{-1}(A_1 \cap A_2) = h^{-1}A_1 \cap h^{-1}A_2 \in \mathcal{F} * \mathcal{G} \\
		&\iff h \in (A_1 \cap A_2)'(\mathcal{F} \sqcap \mathcal{G}).
	\end{align*}

	\smallskip
	
	\textbf{(b):}
	To see statement (i), note
	\[
		h \in (S \setminus A)'(\mathcal{F}) \iff S \setminus (h^{-1}A) = h^{-1}(S \setminus A) \in \mathcal{F} \iff h^{-1}A \not\in \mathcal{F}^* \iff h \not\in A'(\mathcal{F}^*) \iff h \in S \setminus A'(\mathcal{F}^*).
	\]
	
	Now to see statement (ii), note
	\begin{align*}
		h \in (S \setminus A)'(\mathcal{F}) &\iff S \setminus (h^{-1}A) \in \mathcal{F} \\
		&\implies h^{-1}A \not\in \mathcal{F} &\text{(since \( \mathcal{F} \) is a proper filter)} \\
		&\iff h \not\in A'(\mathcal{F}) \iff h \in S \setminus A'(\mathcal{F}).
	\end{align*}
	
	Similarly, to see statement (iii), note
	\begin{align*}
		h \in S \setminus A'(\mathcal{G}) &\iff h^{-1}A \not\in \mathcal{G} \\
		&\implies h^{-1}(S \setminus A) = S \setminus (h^{-1}A) \in \mathcal{G} &\text{(since \( \mathcal{G} \) is a proper grill)} \\
		&\iff h \in (S \setminus A)'(\mathcal{G}).
	\end{align*}	

	Statement (iv) is immediate from (i).
	Since \( A \setminus B = A \cap (S \setminus B) \), the `in particular' statement follows from (a, iii).
	
	\smallskip
	
	\textbf{(c):}
	To see statement (i) note \( h \in A'(\mathcal{F}_1)  \iff h^{-1}A \in \mathcal{F}_1 \subseteq \mathcal{F}_2 \implies h \in A'(\mathcal{F}_2) \).
	
	For statement (ii) note
	\begin{align*}
		\textstyle h \in A'(\bigcup_{i \in I} \mathcal{F}_i) &\iff h^{-1}A \in \bigcup_{i \in I} \mathcal{F}_i\\
        &\iff (\exists \, i \in I) \; h^{-1}A \in \mathcal{F}_i\\
        &\iff (\exists \, i \in I)  \; h \in A'(\mathcal{F}_i)\\
        &\iff h \in \bigcup_{i \in I} A'(\mathcal{F}_i),
	\end{align*}
	and
	\begin{align*}
		\textstyle h \in A'(\bigcap_{i \in I} \mathcal{F}_i) &\iff h^{-1}A \in \bigcap_{i \in I} \mathcal{F}_i\\
        &\iff (\forall \, i \in I) \; h^{-1}A \in \mathcal{F}_i\\
        &\iff (\forall \, i \in I)  \; h \in A'(\mathcal{F}_i)\\
        &\iff h \in \bigcap_{i \in I} A'(\mathcal{F}_i).
	\end{align*}

    For statement (iii) we have \( h \in A_1'(\mathcal{F}) \iff h^{-1}A_1 \in \mathcal{F} \implies h^{-1}A_2 \in \mathcal{F} \iff h \in A_2'(\mathcal{F}) \), where the middle implication follows from \( h^{-1}A_1 \subseteq h^{-1}A_2 \) and \( \mathcal{F} \) is a stack.
    
	\smallskip
	
	\textbf{(d):}
	Note		
	\[
		h \in (g^{-1}A)'(\mathcal{F}) \iff (gh)^{-1}A = h^{-1}(g^{-1}A) \in \mathcal{F} \iff gh \in A'(\mathcal{F}) \iff h \in g^{-1}\bigl(A'(\mathcal{F})\bigr).
	\]
\end{proof}

From Proposition~\ref{proposition:derived-set}(a) and (b) we note for an ultrafilter \( q \) we have \( A \mapsto A'(q) \) is a Boolean algebra endomorphism on \( \mathcal{P}(S) \).
Similarly, for a filter \( \mathcal{F} \) and grill \( \mathcal{G} \) we have \( A \mapsto A'(\mathcal{F}) \) and \( A \mapsto A'(\mathcal{G}) \) are  meet-semilattice and join-semilattice endomorphisms, respectively.
If we instead fix a set \( A \subseteq S \) and allow the collections to vary \( \mathcal{F} \mapsto A'(\mathcal{F}) \) then Proposition~\ref{proposition:derived-set}(c) shows this function is isotone and respects arbitrary unions and intersections.

Next we define a product of collections on \( S \), motivated by the characterization of product of ultrafilters (see Section~\ref{section:review-algebraic-structure}), and prove some basic properties for this product (which are corollaries of Proposition~\ref{proposition:derived-set}). 
Part of the corollary is a generalization of a Berglund and Hindman result \cite[Lemma~5.15]{Berglund1984} showing the binary operation \( \mathcal{F} \cdot \mathcal{G} \) is associative when both \( \mathcal{F}, \mathcal{G} \) are filters.

\begin{definition}
\label{definition:product-of-collections}
	Let \( \mathcal{F}, \mathcal{G} \) both be collections on a semigroup \( S \) and define
	\[
		\mathcal{F} \cdot \mathcal{G} := \{ A \subseteq S : A'(\mathcal{G}) \in \mathcal{F} \}.
	\]
\end{definition}
Following Glasscock and Le \cite[Sec.~1.2.1]{Glasscock:2024aa} we also single out the notions of `translation invariant' and `idempotent' collections (see \cite[Sec.~2.1.4]{Glasscock:2024aa} for why one would consider inclusion instead of equality), which generalizes from idempotent filters and their implicit use in Papazyan \cite{Papazyan1989} and explicit use in Krautzberger \cite{Krautzberger2009}:
\begin{definition}
\label{definition:idempotent-collection}
    Let \( \mathcal{F} \) be a collection on a semigroup \( S \).
    \begin{itemize}
        \item[(a)]
            \( \mathcal{F} \) is \define{translation invariant} if and only if \( \mathcal{F} \subseteq \{S\} \cdot \mathcal{F} \).

        \item[(b)]
             \( \mathcal{F} \) is \define{idempotent} if and only if \( \mathcal{F} \subseteq \mathcal{F} \cdot \mathcal{F} \).
    \end{itemize}
\end{definition}

We directly have translation invariant collections are idempotent (or see Corollary~\ref{corollary:derived-set}(b)).

\begin{example}
\label{example:idempotent-filters}
    Recall the filters \( \mathcal{C}\) , \( \mathcal{H} \), and \( \mathcal{D} \) from Example~\ref{example:filters}.
    As is well known, all three of these filters are idempotent:
    \begin{itemize}
        \item[(a)] 
            As noted in Example~\ref{example:derived-set}(a), \( \mathcal{C} \) is translation invariant and hence idempotent.

        \item[(b)]
            The filter \( \mathcal{H} \) may not be invariant.
            (For instance if every term of the sequence \( \langle x_n \rangle_{n = 1}^\infty \) is an even positive integer, then \( -1 + \mathrm{FS}(\langle x_n \rangle_{n = 1}^\infty) \not\in \mathcal{H} \).)

            However, to see \( \mathcal{H} \) is idempotent note 
            \( 
                A \in \mathcal{H} + \mathcal{H} 
                    \iff 
                (\exists \, k \in \mathbb{N})
                (\forall \, H \in \mathcal{P}_f(\{k, k+1, \dots\}))
                (\exists \, \ell \in \mathbb{N}) 
                    \; \sum_{n \in H} x_n + \mathrm{FS}(\langle x_n \rangle_{n = \ell}^\infty) \subseteq A
            \),
            and observe for \( A \subseteq \mathbb{N} \) such that there exists \( k \in \mathbb{N} \) with \( \mathrm{FS}(\langle x_n \rangle_{n = k}^\infty) \subseteq A \) then for each \( H \in \mathcal{P}_f(\{k, k+1, \ldots\}) \) we have \( \sum_{n \in H} x_n + \mathrm{FS}(\langle x_n \rangle_{n = \max H + 1}^\infty) \subseteq A \).

        \item[(c)]
            Similar to (a), \( \mathcal{D} \) is translation invariant and hence idempotent.            
    \end{itemize}

\end{example}

Observe \( \emptyset \cdot \mathcal{G} = \emptyset \) and \( \mathcal{P}(S) \cdot \mathcal{G} = \mathcal{P}(S) \).
Hence both the improper stacks \( \emptyset \) and \( \mathcal{P}(S) \) are left zeros for this binary operation.
On the other hand we have
\[
	\mathcal{F} \cdot \emptyset = \begin{cases}
		\emptyset &\text{if } \emptyset \notin \mathcal{F} \\
		\mathcal{P}(S) &\text{if } \emptyset \in \mathcal{F}
	\end{cases}
	\qquad \text{ and } \qquad
	\mathcal{F} \cdot \mathcal{P}(S) = \begin{cases}
		\emptyset &\text{if } S \notin \mathcal{F} \\
		\mathcal{P}(S) &\text{if } S \in \mathcal{F}.
	\end{cases}
\]
In particular, this binary operation is not commutative~---~even when the underlying semigroup \( S \) is commutative.

\begin{corollary}
\label{corollary:derived-set}
	Let \( S \) be a semigroup.
	\begin{itemize}
		\item[(a)]
		Let \( \mathcal{F}, \mathcal{G}, \mathcal{H} \) all be collections on \( S \).
		\begin{itemize}
			\item[(i)]
			For \( A \subseteq S \) we have \( \overline{A'(\mathcal{G})} = \{ p \in \beta S : A \in p \cdot \mathcal{G}\} \).
			
			\item[(ii)]
			 For \( A \subseteq S \) we have \( A'(\mathcal{F} \cdot \mathcal{G}) = \bigl( A'(\mathcal{G}) \bigr)'(\mathcal{F}) \).
			 
			\item[(iii)]
			 \( \mathcal{F} \cdot (\mathcal{G} \cdot \mathcal{H}) = (\mathcal{F} \cdot \mathcal{G}) \cdot \mathcal{H} \).		
		\end{itemize}

		\item[(b)]
		If \( \mathcal{F}_1, \mathcal{F}_2, \) and \( \mathcal{G} \) are all collections on \( S \), then \( \mathcal{F}_1 \subseteq \mathcal{F}_2 \) implies \( \mathcal{F}_1 \cdot \mathcal{G} \subseteq \mathcal{F}_2 \cdot \mathcal{G} \).

		\item[(c)]
		If \( \mathcal{F} \) is a stack on \( S \) and both \( \mathcal{G}_1, \mathcal{G}_2 \)  are collections on \( S \), then \(\mathcal{G}_1 \subseteq \mathcal{G}_2 \) implies \( \mathcal{F} \cdot \mathcal{G}_1 \subseteq \mathcal{F} \cdot \mathcal{G}_2 \).
		
		\item[(d)]
		Let \( \mathcal{F}, \mathcal{G} \) both be stacks on \( S \).
		\begin{itemize}
			\item[(i)]
			\( \mathcal{F} \cdot \mathcal{G} \) is a stack, which is proper if both \( \mathcal{F} \) and \( \mathcal{G} \) are proper.
					
			\item[(ii)]
			\( (\mathcal{F} \cdot \mathcal{G})^* = \mathcal{F}^* \cdot \mathcal{G}^* \).

			\item[(iii)]
			If both \( \mathcal{F} \) and \( \mathcal{G} \) are filters on \( S \), then \( \mathcal{F} \cdot \mathcal{G} \) is a filter.
			Dually,  if both \( \mathcal{F} \) and \( \mathcal{G} \) are grills on \( S \), then \( \mathcal{F} \cdot \mathcal{G} \) is a grill.
			In particular, if \( p, q \) are ultrafilters on \( S \), then \( p \cdot q \) is an ultrafilter.

		\end{itemize}

    \item[(e)]
        Let \( \mathcal{F}_1, \mathcal{F}_2, \mathcal{G}_1, \mathcal{G}_2 \)  all be stacks on \( S \).
        \begin{itemize}
            \item[(i)] 
                 \( (\mathcal{F}_1 \cdot \mathcal{G}_1) \sqcap  (\mathcal{F}_2 \cdot \mathcal{G}_2) \subseteq (\mathcal{F}_1 \sqcap \mathcal{F}_2) \cdot (\mathcal{G}_1 \sqcap \mathcal{G}_2) \). 

            \item[(ii)]
                If \( \mathcal{F}_1 \) and \( \mathcal{F}_2 \) are both translation invariant, then \( \mathcal{F}_1 \sqcap \mathcal{F}_2 \) is translation invariant.

            \item[(iii)]
                If \( \mathcal{F}_1 \) and \( \mathcal{F}_2 \) are both idempotent, then \( \mathcal{F}_1 \sqcap \mathcal{F}_2 \) is idempotent.
        \end{itemize}

	\end{itemize}	
\end{corollary}
\begin{proof}
	\textbf{(a):}
	To see statement (i) note \( p \in \overline{A'(\mathcal{G})} \iff A'(\mathcal{G}) \in p \iff A \in p \cdot \mathcal{G} \).

	To see statement (ii) note
	\begin{align*}
		h \in A'(\mathcal{F} \cdot \mathcal{G}) &\iff h^{-1}A \in \mathcal{F} \cdot \mathcal{G} \\
		&\iff h^{-1}\bigl(A'(\mathcal{G})\bigr) = (h^{-1}A)'(\mathcal{G}) \in \mathcal{F} &\text{(equality from Proposition~\ref{proposition:derived-set}(d))} \\
		&\iff h \in \bigl(A'(\mathcal{G})\bigr)'(\mathcal{F}).
	\end{align*}
	
	Finally, to see statement (iii) note
	\begin{align*}
		A \in \mathcal{F} \cdot (\mathcal{G} \cdot \mathcal{H}) &\iff \bigl(A'(\mathcal{H})\bigr)'(\mathcal{G}) = A'(\mathcal{G} \cdot \mathcal{H}) \in \mathcal{F} &\text{(equality by (ii))} \\
		&\iff A'(\mathcal{H}) \in \mathcal{F} \cdot \mathcal{G} \\
		&\iff A \in (\mathcal{F} \cdot \mathcal{G}) \cdot \mathcal{H}.
	\end{align*}
	
	\smallskip
	
	\textbf{(b):}
	We have \( A \in \mathcal{F}_1 \cdot \mathcal{G} \implies A'(\mathcal{G}) \in \mathcal{F}_1 \subseteq \mathcal{F}_2 \implies A \in \mathcal{F}_2 \cdot \mathcal{G} \).
	
	\smallskip

	\textbf{(c):}
	By Proposition~\ref{proposition:derived-set}(c, i) we have \( A'(\mathcal{G}_1) \subseteq A'(\mathcal{G}_2) \).
	Hence from \( A \in \mathcal{F} \cdot \mathcal{G}_1 \) it follows, since \( \mathcal{F} \) is a stack, that \( A'(\mathcal{G}_2) \in \mathcal{F} \), that is, \( A \in \mathcal{F} \cdot \mathcal{G}_2 \).	

	\smallskip

	\textbf{(d):}
	To see \( \mathcal{F} \cdot \mathcal{G} \) is a stack let \( A \in \mathcal{F} \cdot \mathcal{G} \) and \( A \subseteq B \subseteq S \).
	Then \( A'(\mathcal{G}) \in \mathcal{F} \) and since \( \mathcal{G} \) is a stack, by Proposition~\ref{proposition:derived-set}(c, iii), we have \( A'(\mathcal{G}) \subseteq B'(\mathcal{G}) \).
	Now \( \mathcal{F} \) a stack implies \( B'(\mathcal{G}) \in \mathcal{F} \), that is, \( B \in \mathcal{F} \cdot \mathcal{G} \).
	For properness, we have \( \emptyset \ne \mathcal{F} \cdot \mathcal{G} \) since \( S'(\mathcal{G}) = S \in \mathcal{F} \) (proper stack implies \( S \) in both \( \mathcal{F}, \mathcal{G} \)), and \( \emptyset \notin \mathcal{F} \cdot \mathcal{G} \) since \( \emptyset'(\mathcal{G}) = \emptyset \notin \mathcal{F} \).	
	
	To see statement (ii) note
	\begin{align*}
		A \in (\mathcal{F} \cdot \mathcal{G})^* &\iff S \setminus A \not\in \mathcal{F} \cdot \mathcal{G} &\text{((i) and Proposition~\ref{proposition:stack}(b))}\\
		&\iff S \setminus A'(\mathcal{G}^*) = (S \setminus A)'(\mathcal{G}) \not\in \mathcal{F} &\text{(equality by Proposition~\ref{proposition:derived-set}(b, i))}\\
		&\iff A'(\mathcal{G}^*) \in \mathcal{F}^* \iff A \in \mathcal{F}^* \cdot \mathcal{G}^*. &\text{(Proposition~\ref{proposition:stack}(b))}
	\end{align*}

	For statement (iii), by Proposition~\ref{proposition:filter-grill}(a) and (ii), the first and second assertions are equivalent. 
	Hence it suffices to only show one, say the first one.
	The `in particular' assertion then immediately follows.
	
	The first assertion was shown, for proper filters, by Berglund and Hindman \cite[Lemma~5.15]{Berglund1984} but the same argument works for either \( \mathcal{F} \) or \( \mathcal{G} \) improper.
	(Let \( A_1, A_2 \in \mathcal{F} \cdot \mathcal{G} \).
	Then \( A_1'(\mathcal{G}), A_2'(\mathcal{G}) \in \mathcal{F} \) and so, since \( \mathcal{G} \) and \( \mathcal{F} \) are filters, by Proposition~\ref{proposition:derived-set}(a, i), we have \( (A_1 \cap A_2)'(\mathcal{G}) = A_1'(\mathcal{G}) \cap A_2'(\mathcal{G}) \in \mathcal{F} \).)

    \smallskip

    \textbf{(e):}
    First, note by statement (d, i) and Proposition~\ref{proposition:grill}(a, i) we have \( \mathcal{F}_1 \cdot \mathcal{G}_1 \), \( \mathcal{F}_2 \cdot \mathcal{G}_2 \), \( \mathcal{F}_1 \sqcap \mathcal{F}_2 \), and \( \mathcal{G}_1 \sqcap \mathcal{G}_2 \) are all stacks.

    To see statement (i), let \( A_1 \in \mathcal{F}_1 \cdot \mathcal{G}_1 \) and \( A_2 \in \mathcal{F}_2 \cdot \mathcal{G}_2 \).
    Then \( A_1'(\mathcal{F}_1) \in \mathcal{G}_1 \) and \( A_2'(\mathcal{G}_2) \) implies, by Proposition~\ref{proposition:derived-set}(a, iv), that \( A_1'(\mathcal{F}_1) \cap  A_2'(\mathcal{G}_2) \subseteq (A_1 \cap A_2)'(\mathcal{G}_1 \sqcap \mathcal{G}_2) \).
    Hence since \(  A_1'(\mathcal{F}_1) \cap  A_2'(\mathcal{G}_2) \in \mathcal{F}_1 \sqcap \mathcal{F}_2 \) and  \( \mathcal{F}_1 \sqcap \mathcal{F}_2 \) is a stack we have \( (A_1 \cap A_2)'(\mathcal{G}_1 \sqcap \mathcal{G}_2) \in \mathcal{F}_1 \sqcap \mathcal{F}_2 \), that is, \( A_1 \cap A_2 \in (\mathcal{F}_1 \sqcap \mathcal{F}_2) \cdot (\mathcal{G}_1 \sqcap \mathcal{G}_2) \).

    To see statement (ii), note we have \( \mathcal{F}_1 \sqcap \mathcal{F}_2 \subseteq (\{S\} \cdot \mathcal{F}_1) \sqcap (\{S\} \cdot \mathcal{F}_2) \subseteq (\{S\} \sqcap \{S\}) \cdot (\mathcal{F}_1 \sqcap \mathcal{F}_2) = \{S\} \cdot (\mathcal{F}_1 \sqcap \mathcal{F}_2) \).

    To see statement (iii), note we have  \( \mathcal{F}_1 \sqcap \mathcal{F}_2 \subseteq (\mathcal{F}_1  \cdot \mathcal{F}_1) \sqcap (\mathcal{F}_2  \cdot \mathcal{F}_2) \subseteq (\mathcal{F}_1 \sqcap \mathcal{F}_2) \cdot (\mathcal{F}_1 \sqcap \mathcal{F}_2) \).
	
\end{proof}

\section{Characterizations of relative notion via derived sets}
\label{section:relative-notions}

This section's goal is to characterize relative notions of syndetic and thick sets (Definition~\ref{definition:relative-syndetic-thick}) in terms of derived sets along (ultra)filters (Theorem~\ref{theorem:relative-syndetic-thick} and Corollary~\ref{corollary:relative-syndetic-thick}).
Building on this characterization we then, under certain natural conditions, characterize relative notions of piecewise syndetic (Definition~\ref{definition:relative-piecewise-syndetic}) in terms of derived sets along ultrafilters (Theorem~\ref{theorem:relative-piecewise-syndetic}).  
These notions generalize the well studied ``absolute'' notions of syndetic, thick, and piecewise syndetic.
However, it is not a generalization in the sense that syndetic sets are guaranteed to be ``relatively'' syndetic (Example~\ref{example:product-filters}(b)), but rather that being syndetic (in the semigroup \(S\)) relative to the filter \( \{S\} \) is equivalent to being syndetic in \(S \). 
As an application we define a relative ``kernel'' \( K(\mathcal{F}, \mathcal{G}) \) (Definition~\ref{definition:relative-kernel}) and prove this set is, in some ways, analogous to the smallest ideal \( K(\beta S) \) (Theorem~\ref{theorem:relative-kernel} and Corollary~\ref{corollary:relative-kernel}).
We also state a problem on finding ultrafilters which are minimal or maximal with respect to preorders defined in terms of derived sets (Problem~\ref{problem:characterization-problem}).

\subsection{Relative syndetic and thick sets}
\label{section:relative-syndetic-thick}

\begin{definition}
\label{definition:relative-syndetic-thick} %
	Let \( S \) be a semigroup, let \( \mathcal{F}, \mathcal{G} \) both be collections on \( S \), and let \( A \subseteq S \).
	\begin{itemize}
		\item[(a)]
			\( A \) is \define{\( (\mathcal{F}, \mathcal{G}) \)-syndetic} if and only if for all \( B \in \mathcal{F} \) there exists \( H \in \mathcal{P}_f(B) \) such that \[ \bigcup_{h \in H} h^{-1}A \in \mathcal{G}^{**}. \]
			
		\item[(b)] 
			\( A \) is \define{\( (\mathcal{F}, \mathcal{G}) \)-thick} if and only if there exists \( B \in \mathcal{F} \) such that for all \( H \in \mathcal{P}_f(B) \) we have \[ \bigcap_{h \in H} h^{-1}A \in~\mathcal{G}^*. \]

		\item[(c)]
			We define the following collections:
		\begin{itemize}
			\item[(i)] 
				\( \mathsf{Syn}(\mathcal{F}, \mathcal{G}) := \{ A \subseteq S : A \text{ is } (\mathcal{F}, \mathcal{G})\text{-syndetic} \} \).
			
			\item[(ii)]
				\( \mathsf{Thick}(\mathcal{F}, \mathcal{G}) := \{ A \subseteq S : A \text{ is } (\mathcal{F}, \mathcal{G})\text{-thick} \} \).
				
		\end{itemize}
	\end{itemize}
\end{definition}

Shuungula, Zelenyuk, and Zelenyuk \cite[Section~2]{Shuungula:2009ty} defined the notation and name of `\( (\mathcal{F}, \mathcal{G}) \)-syndetic' for proper filters \( \mathcal{F} \) and \( \mathcal{G} \).
The notion of `\( (\mathcal{F}, \mathcal{G}) \)-thick' was introduced in \cite{Christopherson:2022wr}, mainly to ensure, when \( \mathcal{F}, \mathcal{G} \) are proper stacks, \( \mathsf{Thick}(\mathcal{F}, \mathcal{G}) = \mathsf{Syn}(\mathcal{F}, \mathcal{G})^* \) and directly obtain the collection of relative piecewise syndetic sets (see Definition~\ref{definition:relative-piecewise-syndetic}) form a grill (following from Proposition~\ref{proposition:grill}(b)). 
We soon note in Proposition~\ref{proposition:relative-syndetic-thick} this duality between relative syndetic and thick sets via the mesh also holds in a slightly more general context.
However, we first note the easy observation that we may assume either \( \mathcal{F} \) or \( \mathcal{G} \) is a stack:
\begin{proposition}
\label{proposition:assumption-of-stack}
    Let \( \mathcal{F}, \mathcal{G} \) both be collections on a semigroup \( S \).
    \begin{itemize}
        \item[(a)]
            We have \( \mathsf{Syn}(\mathcal{F}, \mathcal{G}) =  \mathsf{Syn}(\mathcal{F}^{**}, \mathcal{G}) \) and \( \mathsf{Syn}(\mathcal{F}, \mathcal{G}) =  \mathsf{Syn}(\mathcal{F}, \mathcal{G}^{**}) \).
            In particular we have \( \mathsf{Syn}(\mathcal{F}, \mathcal{G}) =  \mathsf{Syn}(\mathcal{F}^{**}, \mathcal{G}^{**}) \).

        \item[(b)]
            We have \( \mathsf{Thick}(\mathcal{F}, \mathcal{G}) =  \mathsf{Thick}(\mathcal{F}^{**}, \mathcal{G}) \) and \( \mathsf{Thick}(\mathcal{F}, \mathcal{G}) =  \mathsf{Thick}(\mathcal{F}, \mathcal{G}^{**}) \).
            In particular we have \( \mathsf{Thick}(\mathcal{F}, \mathcal{G}) =  \mathsf{Thick}(\mathcal{F}^{**}, \mathcal{G}^{**}) \).
    \end{itemize}
\end{proposition}
\begin{proof}
    First, note membership in \( \mathsf{Syn}(\mathcal{F}, \mathcal{G}) \) or \( \mathsf{Thick}(\mathcal{F}, \mathcal{G}) \) is determined as follows:
    \[
        A \in \mathsf{Syn}\left(\f{F},\f{G}\right) \iff \left(\forall \, B \in \f{F}\right) \left(\exists \, H \in \f{P}_f\left(B\right)\right) \left( \exists \, C \in \f{G} \right) \left(\forall \,c \in C\right) \left(\exists \,h \in H\right) \; hc \in A
    \] 
    and
    \[
        A \in \mathsf{Thick}\left(\f{F},\f{G}\right) \iff \left(\exists \, B \in \f{F}\right) \left(\forall \, H \in \f{P}_f\left(B\right)\right) \left( \forall \, C \in \f{G} \right) \left(\exists \,c \in C\right) \left(\forall \,h \in H\right) \; hc \in A.
    \]
    Then the assertions follow from  Proposition~\ref{proposition:stack}(a) and (c), and Corollary~\ref{corollary:mesh-operator}(c).

\end{proof}

Observe \( \mathsf{Syn}(\emptyset, \mathcal{G}) = \mathcal{P}(S) \), \( \mathsf{Syn}(\mathcal{P}(S), \mathcal{G}) = \mathsf{Syn}(\{\emptyset\}, \mathcal{G}) = \emptyset \), while
\[
	\mathsf{Syn}(\mathcal{F}, \emptyset) = \begin{cases}
		\emptyset &\text{if } \emptyset \ne \mathcal{F} \\
		\mathcal{P}(S) &\text{if } \emptyset = \mathcal{F}
	\end{cases}
	\qquad \text{ and } \qquad
	\mathsf{Syn}(\mathcal{F}, \mathcal{P}(S)) = \begin{cases}
		\emptyset &\text{if } \emptyset \in \mathcal{F} \\
		\mathcal{P}(S) &\text{if } \emptyset \notin \mathcal{F}
	\end{cases}.
\]
Hence if \( \mathcal{F} \) is a proper collection, then \( \mathsf{Syn}(\mathcal{F}, \emptyset) = \emptyset \) and \( \mathsf{Syn}(\mathcal{F}, \mathcal{P}(S)) = \mathcal{P}(S) \).
\begin{proposition}
\label{proposition:relative-syndetic-thick}
	Let \( S \) be a semigroup and let both \( \mathcal{F}, \mathcal{G} \) be collections.
	\begin{itemize}
		\item[(a)]
		\( \mathsf{Syn}(\mathcal{F}, \mathcal{G}) \) is a stack on \( S \), which is proper if both \( \mathcal{F} \) and \( \mathcal{G} \) are proper. 

		\item[(b)]
		\( \mathsf{Syn}(\mathcal{F}, \mathcal{G})^* =  \mathsf{Thick}(\mathcal{F}, \mathcal{G}) \).
		In particular, \( \mathsf{Thick}(\mathcal{F}, \mathcal{G}) \) is a stack on \( S \).

	\end{itemize}	
\end{proposition}
\begin{proof}
   	\textbf{(a):}
    To see \( \mathsf{Syn}(\mathcal{F}, \mathcal{G}) \) is a stack let \( A_1 \in \mathsf{Syn}(\mathcal{F}, \mathcal{G}) \) and let  \( A_1 \subseteq A_2 \subseteq S \).
    For \( B \in \mathcal{F} \) pick \( H \in \mathcal{P}_f(B) \) with \( \bigcup_{h \in H} h^{-1}A_1 \in \mathcal{G}^{**} \).
    Then \( \mathcal{G}^{**} \) is a stack (Proposition~\ref{proposition:stack}(a)) and \( \bigcup_{h \in H} h^{-1}A_1 \subseteq \bigcup_{h \in H} h^{-1}A_2 \) implies \( A_2 \in \mathsf{Syn}(\mathcal{F}, \mathcal{G}^{**}) = \mathsf{Syn}(\mathcal{F}, \mathcal{G}) \), where the equality follows from Proposition~\ref{proposition:assumption-of-stack}(a).

    Properness for \( \mathcal{F} \) and \( \mathcal{G} \) implies \( S \in \mathsf{Syn}(\mathcal{F}, \mathcal{G}) \) (so \( \emptyset \ne \mathsf{Syn}(\mathcal{F}, \mathcal{G}) \)) and \( \emptyset \notin \mathsf{Syn}(\mathcal{F}, \mathcal{G}) \).

    \smallskip

	\textbf{(b):}
    By Proposition~\ref{proposition:assumption-of-stack} we may replace \( \mathcal{F}, \mathcal{G} \) by their stack closures \( \mathcal{F}^{**}, \mathcal{G}^{**} \) (see Remark~\ref{remark:closure-operator}).
    The same (easy) verification of \cite[Proposition~3.2]{Christopherson:2022wr}~---~where we assumed both \( \mathcal{F} \) and \( \mathcal{G} \) are proper stacks~---~works in this case too (via Proposition~\ref{proposition:stack}(b)).
	The `in particular' statement follows from Proposition~\ref{proposition:stack}(a).
\end{proof}

\begin{remark}
\label{remark:assumption-of-stack}
    In light of Proposition~\ref{proposition:assumption-of-stack} and to slightly ease the notation, in what follows we will typically assume both \( \mathcal{F}, \mathcal{G} \) are at least stacks.
    For then by Proposition~\ref{proposition:stack}(d) we have \( \mathcal{F} = \mathcal{F}^{**} \) and  \( \mathcal{G} = \mathcal{G}^{**} \).
    
    Hence some of our stated results can be made slightly stronger by only assuming both  \( \mathcal{F}, \mathcal{G} \) are collections. 
    With these weaker assumptions, the remaining assertions and proofs can be made valid by replacing occurrences of `\( \mathcal{G} \)' with `\( \mathcal{G}^{**} \)' while simplifying with \( \mathcal{G}^{***} = \mathcal{G}^* \) (Corollary~\ref{corollary:mesh-operator}(c)) when possible.
\end{remark}

As outlined in \cite[Section~3]{Christopherson:2022wr} special cases of both relative notions appear more-or-less explicitly throughout the literature.
\begin{example}
\label{example:relative-syndetic-thick}
    Recall the filters \( \mathcal{C} \), \( \mathcal{H} \), and \( \mathcal{D} \) from Example~\ref{example:filters}.
    Then, by Proposition~\ref{proposition:relative-syndetic-thick}(b), up to ``mesh duality'' we have at most 9 distinct notions of \( \mathsf{Syn}(\mathcal{F}, \mathcal{G}) \) when \( \mathcal{F}, \mathcal{G} \in \{ \mathcal{C}, \mathcal{H}, \mathcal{D} \} \).
    We'll just highlight a small subset of relative syndetic and thick in this context.
    For filters \( \mathcal{F}, \mathcal{G} \) with corresponding bases \(\f{B_F},\f{B_G}\) (that is, \( \f{F} = \f{B_F}^{**} \) and  \( \f{G} = \f{B_G}^{**} \)) we apply Proposition~\ref{proposition:assumption-of-stack} to just characterize \( \mathsf{Syn}(\mathcal{B}_\mathcal{F}, \mathcal{B}_\mathcal{G}) \).
    In the following, let \( \forall^\infty \) be the ``for all but finitely many'' and \( \exists^\infty \) be the ``there exists infinitely many'' quantifiers.

    \begin{itemize}
        \item[(a)]
           In \( (\mathbb{N}, +) \) we have
            \[
                \mathsf{Syn}\left(\mathcal{C},\mathcal{C}\right) 
                    = 
                    \{ A \subseteq \mathbb{N} : 
                        \left(\forall \, k \in \mathbb{N}\right)
                        \left(\exists \, \ell >k\right) 
                        \left(\forall^{\infty} \, x \in \mathbb{N} \right)
                        \left(\exists \, h \in \left[k, \ell\right]\right) 
                            \; h + x \in A \}
            \]
            and
            \[
                \mathsf{Thick}\left(\mathcal{C},\mathcal{C}\right) 
                = 
                    \{ A \subseteq \mathbb{N} : 
                        \left(\exists \, k \in \mathbb{N}\right)
                        \left(\forall \, \ell > k \right) 
                        \left(\exists^{\infty} \, x \in \mathbb{N} \right) 
                        \left(\forall \, h \in \left[k,\ell\right] \right) 
                            \; h + x \in A\}.
            \]
            As noted in \cite[Example~3.6]{Christopherson:2022wr}, we have \( \mathsf{Thick} = \mathsf{Thick}(\mathcal{C}, \{\mathbb{N}\}) = \mathsf{Thick}(\{\mathbb{N}\}, \mathcal{C}) = \mathsf{Thick}(\mathcal{C}, \mathcal{C}) \), and dually we have \( \mathsf{Syn} = \mathsf{Syn}(\mathcal{C}, \mathcal{C}) \).
            Hence in certain cases the ``relative'' notions are equivalent to the absolute notions.

        \item[(b)]
            We have
            \[
                \begin{split}
                    \mathsf{Syn}(\mathcal{H}, \mathcal{H}) 
                    = 
                    \{ A \subseteq \mathbb{N} : 
                        (\forall \, k \in \mathbb{N}) 
                        &(\exists \, \text{finite sequence } \langle H_i \rangle_{i = 1}^n \text{ in } \mathcal{P}_f(\{k, k+1, \ldots\})) \\
                        &\hspace{-2.5em}(\exists \, m \in \mathbb{N} \text{ with } m > \max {\textstyle \bigcup_{i = 1}^n H_i})
                        (\forall \, G \in \mathcal{P}_f(\{m, m+1, \ldots\})) \\
                        &\hspace{-0.5em}(\exists \, i \in \{1, 2, \ldots, n\}) 
                          \; {\textstyle \sum_{s \in H_i \cup G} x_s \in A} \}.
                \end{split}
             \]
             and
            \[
                \begin{split}
                    \mathsf{Thick}(\mathcal{H}, \mathcal{H}) 
                    = 
                    \{ A \subseteq \mathbb{N} : 
                        (\exists \, k \in \mathbb{N}) 
                        &(\forall \, \text{finite sequence } \langle H_i \rangle_{i = 1}^n \text{ in } \mathcal{P}_f(\{k, k+1, \ldots\})) \\
                        &\hspace{-2.5em}(\forall \, m \in \mathbb{N} \text{ with } m > \max {\textstyle \bigcup_{i = 1}^n H_i})
                        (\exists \, G \in \mathcal{P}_f(\{m, m+1, \ldots\})) \\
                        &\hspace{-0.5em}(\forall \, i \in \{1, 2, \ldots, n\}) 
                          \; {\textstyle \sum_{s \in H_i \cup G} x_s \in A} \}.
                \end{split}
             \]
            In Example~\ref{example:product-filters}(b) we note some sufficient conditions on the sequence \( \langle x_n \rangle_{n = 1}^\infty \) that shows \( \mathsf{Syn}(\mathcal{H}, \mathcal{H}) \ne \mathsf{Syn} \).

            \item[(c)]
                We have
                \[
                    \mathsf{Syn}\left(\mathcal{H},\mathcal{C}\right) 
                    = 
                    \{ A \subseteq \mathbb{N} : 
                        \left(\forall \, k \in \mathbb{N}\right) 
                        \left(\exists \, \ell > k\right) 
                        \left(\forall^{\infty} \, t \in \mathbb{N}\right) 
                        \left(\exists \, H \in \mathcal{P}_f([k, \ell]) \right)
                            \; \textstyle \sum_{n \in H} x_n  + t \in A \}
                \]
                and
                \[
                    \mathsf{Thick}\left(\mathcal{H},\mathcal{C}\right) 
                    = 
                    \{ A \subseteq \mathbb{N} : 
                        \left(\exists \, k \in \mathbb{N}\right) 
                        \left(\forall \, \ell > k\right) 
                        \left(\exists^{\infty} \, t \in \mathbb{N}\right) 
                        \left(\forall \, H \in \mathcal{P}_f([k, \ell]) \right)
                            \; \textstyle \sum_{n \in H} x_n  + t \in A \}.
                \]
                
            \item[(d)]
                Finally, we have
                \[
                \begin{split}
                    \mathsf{Syn}(\mathcal{C}, \mathcal{H}) 
                    = 
                    \{ A \subseteq \mathbb{N} :
                        \left(\forall \, k \in \mathbb{N}\right)
                        \left(\exists \, \ell > k \right)
                        &(\exists \, m \in \mathbb{N}) \\
                        &\hspace{-10em}(\forall \, H \in \mathcal{P}_f(\{m, m + 1, \ldots\})
                        (\exists \, t \in \left[k, \ell\right])
                            \; \textstyle t + \sum_{n \in H} x_n \in A \}.
                \end{split}
                \]
                and                 
                \[
                \begin{split}
                    \mathsf{Thick}(\mathcal{C}, \mathcal{H}) 
                    = 
                    \{ A \subseteq \mathbb{N} :
                        \left(\exists \, k \in \mathbb{N}\right)
                        \left(\forall \, \ell > k \right)
                       &(\forall \, m \in \mathbb{N}) \\
                       &\hspace{-10em} (\exists \, H \in \mathcal{P}_f(\{m, m + 1, \ldots\})
                        (\forall \, t \in \left[k, \ell\right])
                            \; \textstyle t + \sum_{n \in H} x_n \in A \}.
                \end{split}
                \]

    \end{itemize}

\end{example}

A ``characterization problem'' was stated in \cite[Problem~3.8]{Christopherson:2022wr}  which asked 
to use the algebraic structure of \( \beta S \) to characterize \( (\mathcal{F}, \mathcal{G}) \)-syndetic and  \( (\mathcal{F}, \mathcal{G}) \)-thick for stacks \( \mathcal{F}, \mathcal{G} \).
The next result, in a sense, reduces this problem to studying certain families of products of a collection with a stack.
It generalizes \cite[Theorem~4.4]{Griffin:2024aa} and, as a corollary, we can obtain the special cases proved earlier in \cite[Lemma~3.9]{Christopherson:2022wr}:
\begin{theorem}
\label{theorem:relative-syndetic-thick}	
	Let \( S \) be a semigroup and let both \( \mathcal{F}, \mathcal{G} \) both be stacks.
	\begin{itemize}
		\item[(a)]
			\( \mathsf{Syn}(\mathcal{F}, \mathcal{G}) = \{ A \subseteq S : \text{for every maximal filter } \mathcal{H} \subseteq \mathcal{G}^* \text{ we have } A'(\mathcal{H}^*) \in \mathcal{F}^* \}  \).
			
		\item[(b)]
			\( \mathsf{Thick}(\mathcal{F}, \mathcal{G}) = \{ A \subseteq S : \text{there exists a maximal filter } \mathcal{H} \subseteq \mathcal{G}^* \text{ such that } A'(\mathcal{H}) \in \mathcal{F} \}  \).
	\end{itemize}
\end{theorem}
\begin{proof} 
	First, if either \( \mathcal{F} \) or \( \mathcal{G} \) are improper, then by the above observation on \( \mathsf{Syn}(\mathcal{F}, \mathcal{G}) \) in these improper cases, statement (a) is satisfied. 
	(Note the only maximal filter in \( \mathcal{P}(S) = \emptyset^* \) is the improper filter \( \mathcal{P}(S) \) itself.)
	So, we may assume both \( \mathcal{F} \) and \( \mathcal{G} \) are proper.
	Under the assumption \( \mathcal{G} \) is proper, we have \( \mathcal{G}^* \) is proper too (Proposition~\ref{proposition:stack}(a)) and all (maximal) filters contained in \( \mathcal{G}^* \) are necessarily proper. 
	
	Second, observe by Propositions~\ref{proposition:relative-syndetic-thick} and \ref{proposition:derived-set}(b, i) it follows that statements (a) and (b) are equivalent.
	Hence it suffices to only show one of them, say statement (b).
	
	\smallskip
		
	To this end, we first show if \( A \in \mathsf{Thick}(\mathcal{F}, \mathcal{G}) \), then there exists a filter \( \mathcal{H} \subseteq \mathcal{G}^* \) with \( A'(\mathcal{H}) \in \mathcal{F} \).
	
	Let \( A \in \mathsf{Thick}(\mathcal{F}, \mathcal{G}) \) and pick \( B \in \mathcal{F} \) as guaranteed for \( A \).
	Then \( \{ \bigcap_{h \in H} h^{-1}A : H \in \mathcal{P}_f(B) \} \) is a filterbase contained in \( \mathcal{G}^* \).
	(Since \( \bigl(\bigcap_{h \in H_1} h^{-1}A\bigr) \cap \bigl(\bigcap_{h \in H_2} h^{-1}A\bigr) = \bigcap_{h \in H_1 \cup H_2} h^{-1}A \).)
	Therefore \( \mathcal{H} := \{ \bigcap_{h \in H} h^{-1}A : H \in \mathcal{P}_f(B) \}^{**} \) (see Proposition~\ref{proposition:stack}(c)) is a filter and because \( \mathcal{G}^* \) is a stack it follows that \( \mathcal{H} \subseteq \mathcal{G}^* \).
	Now \( \{ h^{-1}A : h \in B \} \subseteq \mathcal{H} \), that is, \( B \subseteq A'(\mathcal{H}) \).
	Since \( \mathcal{F} \) is a stack we have \( A'(\mathcal{H}) \in \mathcal{F} \).
	
	We now show if \( A \subseteq S \) such that there exists a filter \( \mathcal{H} \subseteq \mathcal{G}^* \) with \( A'(\mathcal{H}) \in \mathcal{F} \), then there exists a \emph{maximal} proper filter \( \mathcal{H} \subseteq \mathcal{G}^* \) with \( A'(\mathcal{H}) \in \mathcal{F} \).

	This follows from a direct application of Zorn's lemma.
	By assumption, the set \[ \Phi := \{ \mathcal{H} : \mathcal{H} \subseteq \mathcal{G}^* \text{ is a filter and } A'(\mathcal{H}) \in \mathcal{F} \} \] is nonempty and partially ordered by inclusion.
	Given a chain in \( \Phi \), the union of the chain is a filter contained in \( \mathcal{G}^* \) and, by Propositions~\ref{proposition:derived-set}(c, ii) and (c, iii) we have the union of the chain is also in \( \Phi \) (where we use the fact \( \mathcal{F} \) is a stack).
	Hence every chain in the collection has an upper bound and we can apply Zorn's lemma to conclude there is a maximal proper filter.
	
	Finally, we show if \( A \subseteq S \) such that there exists a (maximal proper) filter \( \mathcal{H} \subseteq \mathcal{G}^* \) with \( A'(\mathcal{H}) \in \mathcal{F} \), then \( A \in \mathsf{Thick}(\mathcal{F}, \mathcal{G}) \).
	Putting \( B := A'(\mathcal{H}) \) for all \( H \in \mathcal{P}_f(B) \) we have, since \( \mathcal{H} \) is a filter, \( \bigcap_{h \in H} h^{-1}A \in \mathcal{H} \subseteq \mathcal{G}^* \).	
\end{proof}

Since \( A'(\mathcal{H}) \in \mathcal{F} \iff A \in \mathcal{F} \cdot \mathcal{H} \) we have, by Theorem~\ref{theorem:relative-syndetic-thick}, that for stacks \( \mathcal{F} \) and \( \mathcal{G} \)
\[
    \mathsf{Syn}(\mathcal{F}, \mathcal{G}) = \bigcap_{\substack{\mathcal{H} \subseteq \mathcal{G}^* \\ \mathcal{H} \text{ is a filter}}} \mathcal{F}^* \cdot \mathcal{H}^*     
    \quad \text{ and } \quad
    \mathsf{Thick}(\mathcal{F}, \mathcal{G}) = \bigcup_{\substack{\mathcal{H} \subseteq \mathcal{G}^* \\ \mathcal{H} \text{ is a filter}}} \mathcal{F} \cdot \mathcal{H}.
\]
Hence one interpretation of Theorem~\ref{theorem:relative-syndetic-thick} is that relative notions of syndetic and thick naturally arise when considering the product of a stack with either a grill or filter in the second component of the product.
Note that even when \( \mathcal{G} \) is a proper stack, maximal filters contained in \( \mathcal{G}^* \) are \emph{not necessarily} ultrafilters.
(For instance, consider \( \{2\mathbb{N}, 2\mathbb{N}-1\}^{**} \) the union of two filters, one generated by the even and the other by odd positive integers.)
However, if \( \mathcal{G} \) is a filter, then maximal filters contained in \( \mathcal{G}^* \) are precisely the ultrafilters in \( \overline{\mathcal{G}} \).
(Of course, if \( \mathcal{G} = \mathcal{P}(S) \) is improper, then \( \overline{\mathcal{G}} = \emptyset \).)
Hence restricting the second component \( \mathcal{G} \) to be a filter we obtain the following characterizations of relative syndetic and thick in terms of derived sets along ultrafilters:
\begin{corollary}
\label{corollary:relative-syndetic-thick}
	Let \( S \) be a semigroup, let \( \mathcal{F} \) be a stack, and let \( \mathcal{G} \) be a filter on \( S \).
	\begin{itemize}
		\item[(a)]
		\( \mathsf{Syn}(\mathcal{F}, \mathcal{G}) = \{A \subseteq S : \text{for all } q \in \overline{\mathcal{G}} \text{ we have } A'(q) \in \mathcal{F}^* \}\) and \( \mathsf{Syn}(\mathcal{F}, \mathcal{G}^*) = \mathcal{F}^* \cdot \mathcal{G}^* \).
	
		\item[(b)]
		\( \mathsf{Thick}(\mathcal{F}, \mathcal{G}) = \{A \subseteq S : \text{there exists } q \in \overline{\mathcal{G}} \text{ such that } A'(q) \in \mathcal{F} \}\) and \( \mathsf{Thick}(\mathcal{F}, \mathcal{G}^*) = \mathcal{F} \cdot \mathcal{G} \).	
	\end{itemize}
\end{corollary}
\begin{proof}
	Similar to the observation in the second paragraph of the proof of Theorem~\ref{theorem:relative-syndetic-thick}, statements (a) and (b) are equivalent.
	Hence it only suffices to show one of them, say statement (b).
	
	\smallskip
	The first assertion follows directly from Theorem~\ref{theorem:relative-syndetic-thick} since maximal filters contained in \( \mathcal{G}^* \) are precisely ultrafilters in \( \overline{\mathcal{G}} \).
	For the second assertion we have, by Proposition~\ref{proposition:stack}(d), \( \mathcal{G}^{**} = \mathcal{G} \) and the largest filter contained in \( \mathcal{G} \) is simply \( \mathcal{G} \) itself.
\end{proof}

We note (but do not stop to prove here) that if \( \mathcal{F} \) is also a proper filter or proper grill and \( \mathcal{G} \) is assumed proper, then \cite[Lemma~3.9]{Christopherson:2022wr} follows from Corollary~\ref{corollary:relative-syndetic-thick}.
Given the context of Corollary~\ref{corollary:relative-syndetic-thick} we can state another new aspect of the ``characterization problem'' as determining minimal and maximal ultrafilters with respect to certain preorders, by which we mean a transitive and reflexive relation:
\begin{problem}
\label{problem:characterization-problem}
	Let \( S \) be a semigroup and let \( \mathcal{F} \) be a proper stack.
	Define the relation \( \lesssim_{(\mathcal{F})} \) on \( \beta S \) by \( p \lesssim_{(\mathcal{F})} q \) if and only if \( \mathcal{F} \cdot p \subseteq \mathcal{F} \cdot q \), that is, for all \( A \subseteq S \) we have \( A'(\mathcal{F}) \in p \) implies \( A'(\mathcal{F}) \in q  \).
	Corollary~\ref{corollary:derived-set}(d) guarantees \( \lesssim_{(\mathcal{F})} \) is a preorder on \( \beta S \).
	(Similarly, we can define a preorder where the second component in the product is fixed.)
	Determine, with respect to the preorder \( \lesssim_{(\mathcal{F})} \),  minimal and maximal ultrafilters both in \( \beta S \) and, more generally in \( \overline{\mathcal{G}} \) for a proper filter \( \mathcal{G} \).

\end{problem}

    Luperi Baglini \cite[Sec.~5]{LuperiBaglini2014} asks a similar question on determining minimal ultrafilters with respect to a related preorder that is defined in terms a notion of ``finite embeddability'' of sets and ultrafilters due to Blass and Di Nasso \cite{Blass2016}.
    (Following \cite[Theorems~4 and 11]{Blass2016}, one can define, for \( A, B \subseteq S \), that \( A \) is \define{finitely embeddable} in \( B \) if there exists \( q \in \beta S \) with \( A \subseteq B'(q) \), while for \( p, q \in \beta S \) we say \( p \) is \define{finitely embeddable} in \( q \) if \( q \in c\ell(p \cdot \beta S) \).)
    Šobot \cite[Sec.~3]{v-Sobot:2022aa} shows the existence of certain minimal sets and ultrafilters with respect to finite embeddability in positive integers under multiplication, while Luperi Baglini \cite[Theorem~4.13]{Luperi-Baglini:2016aa} proves maximal ultrafilters with respect to finite embeddability in an arbitrary semigroup \(S \) are precisely elements in the closure of the smallest ideal \( c\ell \bigl( K(\beta S) \bigr) \).
    Analogously, we show, with an additional assumption on \( \mathcal{F} \), that we can characterize \( \lesssim_{(\mathcal{F})} \) maximal ultrafilters in a certain closed subset, but we don't know the situation for weaker assumptions on \( \mathcal{F} \) nor (even any partial) results concerning minimal ultrafilters.

\begin{theorem}
\label{theorem:maximal-elements}
    Let \( S \) be a semigroup, let \( \mathcal{F} \) be a proper filter on \( S \) with \( \mathcal{F} \subseteq \mathsf{Syn}(\mathcal{F}^*, \mathcal{F}) \), and let \( q \in \beta S \) with \( \overline{\mathcal{F}} \cdot q \subseteq \overline{\mathcal{F}} \).
    The following are equivalent.
    \begin{itemize}
        \item[(a)] \( q \) is \( \lesssim_{(\mathcal{F})} \) maximal among ultrafilters in \( \overline{\mathcal{F}} \).

        \item[(b)] \( \overline{\mathcal{F}} \cdot q \) is a minimal left ideal of \( \overline{\mathcal{F}} \).
    \end{itemize}
\end{theorem}
\begin{proof}
    By \cite[Corollary~3.12(a)]{Christopherson:2022wr} the condition \( \mathcal{F} \subseteq \mathsf{Syn}(\mathcal{F}^*, \mathcal{F}) \) is equivalent to stating \( \overline{\mathcal{F}} \) is a closed subsemigroup of \( \beta S \).
    Now by Corollary~\ref{corollary:derived-set}(d, iii) we have \( \mathcal{F} \cdot p \) is a filter for every \( p \in \beta S \), and it follows (or see \cite[Theorem~4.3]{Griffin:2024aa}) that \( \overline{\mathcal{F} \cdot p} = \overline{\mathcal{F}} \cdot p \).

    \smallskip

    \textbf{(a) \( \Rightarrow \) (b):}
    It follows that \( \overline{\mathcal{F}} \cdot q \) is a left ideal of \( \overline{\mathcal{F}} \). 
    Let \( L \) be a left ideal of \( \overline{\mathcal{F}} \) with \( L \subseteq \overline{\mathcal{F}} \cdot q \).
    Pick \( p \in L \) and note \( \overline{\mathcal{F} \cdot p} = \overline{\mathcal{F}} \cdot p \subseteq L \subseteq \overline{\mathcal{F}} \cdot q = \overline{\mathcal{F} \cdot q} \) implies \( \mathcal{F} \cdot q \subseteq \mathcal{F} \cdot p \).
    By maximality of \( q \) we have \( \mathcal{F} \cdot q = \mathcal{F} \cdot p \) and so \( \overline{\mathcal{F}} \cdot p = L = \overline{\mathcal{F}} \cdot q \).

    \smallskip

    \textbf{(b) \( \Rightarrow \) (a):}
    Let \( p \in \overline{\mathcal{F}} \) with  \( q \lesssim_{(\mathcal{F})} p \), that is, \( \mathcal{F} \cdot q \subseteq \mathcal{F} \cdot p \). 
    Then \( \overline{\mathcal{F}} \cdot p \subseteq \overline{\mathcal{F}} \cdot q \) and, by minimality, we have \( \overline{\mathcal{F}} \cdot p = \overline{\mathcal{F}} \cdot q \).
    Hence \( \mathcal{F} \cdot q = \mathcal{F} \cdot p \) and so \( q \equiv_{(\mathcal{F})} p \).
\end{proof}

Since the condition \( \mathcal{F} \subseteq \mathsf{Syn}(\mathcal{F}^*, \mathcal{F}) \) (provided \( \mathcal{F} \) is a proper filter) is equivalent to asserting \( \overline{\mathcal{F}} \) is a closed subsemigroup of \( \beta S \), we follow \cite[Section~2]{Di-Nasso:2018aa} in giving such filters a special name:
\begin{definition}
\label{definition:product-filters}  
    A filter \( \mathcal{F} \) is a \define{product filter} (or \define{additive filter}, if \( S \) is commutative) if and only if \( \mathcal{F} \subseteq \mathsf{Syn}(\mathcal{F}^*, \mathcal{F}) \).
\end{definition}

Analogous to Corollary~\ref{corollary:derived-set}(e, ii) and (e, iii), we have the following well known result that the ``join'' of two product filters is a product filter:
\begin{proposition}
\label{proposition:product-filters}
    If \( \mathcal{F}_1, \mathcal{F}_2 \) are both proper product filters with \( \mathcal{F}_1 \subseteq \mathcal{F}_2^* \), then \( \mathcal{F}_1 \sqcap \mathcal{F}_2 \) is a product filter.
\end{proposition}
\begin{proof}
    First, we note that \( \mathcal{F}_1 \sqcap \mathcal{F}_2 \) is the smallest filter that contains both \( \mathcal{F}_1 \) and \( \mathcal{F}_2 \).
    (Proposition~\ref{proposition:grill} shows \( \mathcal{F}_1 \sqcap \mathcal{F}_2 \) is a stack, while closure under finite intersections follows from definition, and the assumption of properness and \( \mathcal{F}_1 \subseteq \mathcal{F}_2^* \) we have \(  \mathcal{F}_1 \sqcap \mathcal{F}_2 \) is proper.
    Finally if \( \mathcal{F} \) is a filter with \( \mathcal{F}_1 \cup \mathcal{F}_2 \subseteq \mathcal{F} \), then it follows that \( \mathcal{F}_1 \sqcap \mathcal{F}_2 \subseteq \mathcal{F} \).)

    Therefore we have \( \emptyset \ne \overline{\mathcal{F}_1 \sqcap \mathcal{F}_2} = \overline{\mathcal{F}_1} \cap \overline{\mathcal{F}_2} \).
    Since nonempty intersections of closed subsemigroups of \( \beta S \) is a closed subsemigroup of \( \beta S \), we have \( \overline{\mathcal{F}_1 \sqcap \mathcal{F}_2} \) is a closed subsemigroup of \( \beta S \), that is, \( \mathcal{F}_1 \sqcap \mathcal{F}_2 \) is a product filter.
\end{proof}

\begin{example}
\label{example:product-filters}
    \mbox{}
    \begin{itemize}
        \item[(a)] 
            As already known, all proper idempotent filters are product filters.
            (By Corollary~\ref{corollary:relative-syndetic-thick} we have \( \mathsf{Syn}(\mathcal{F}^*, \mathcal{F}) = \bigcap_{q \in \overline{\mathcal{F}}} \mathcal{F} \cdot q \) and Corollary~\ref{corollary:derived-set}(c) implies \( \mathcal{F} \cdot \mathcal{F} \subseteq \mathcal{F} \cdot q \) for each \( q \in \overline{\mathcal{F}} \).)
            In particular in \( (\mathbb{N}, +) \) all of the filters in Example~\ref{example:filters} are idempotent (Example~\ref{example:idempotent-filters}) and hence are product filters.
            As is also known, there are product filters that are not idempotent filters (for instance, see \cite[Example~3.9]{Di-Nasso:2018aa}).

        \item[(b)]
            If we pick a sequence \( \langle x_n \rangle_{n = 1}^\infty \) in \( \mathbb{N} \) such that for each \( n \in \mathbb{N} \) we have \( x_{n + 1} > \sum_{t = 1}^n x_t \) and \( \{  x_{n + 1} - \sum_{t = 1}^n x_t : n \in \mathbb{N} \} \) is unbounded, then by a result of Adams, Hindman, and Strauss \cite[Corollary~4.2]{Adams:2008aa} we have each \( \mathrm{FS}(\langle x_n \rangle_{n = m}^\infty) \) is \emph{not} piecewise syndetic and so \( \overline{\mathcal{H}} \cap c\ell\bigl( K(\mathbb{N}) \bigr) = \emptyset \).
            From part (a) \( \mathcal{H} \) is a product filter and so \( \mathcal{H} \subseteq \mathsf{Syn}(\mathcal{H}^*, \mathcal{H}) \subseteq \mathsf{Syn}(\mathcal{H}, \mathcal{H}) \).
            In particular there are \( (\mathcal{H}, \mathcal{H}) \)-syndetic sets which are not piecewise syndetic.          
    \end{itemize}
\end{example}

\subsection{Relative piecewise syndetic sets}
\label{section:relative-piecewise-syndetic}

\begin{definition}
\label{definition:relative-piecewise-syndetic}
	Let \( S \) be a semigroup and let \( \mathcal{F}, \mathcal{G} \) both be collections on \( S \).
	Define the collection \( \mathsf{PS}(\mathcal{F}, \mathcal{G}) := \mathsf{Syn}(\mathcal{F}, \mathcal{G}) \sqcap \mathsf{Thick}(\mathcal{F}, \mathcal{G}) \).
	The elements of \( \mathsf{PS}(\mathcal{F}, \mathcal{G}) \) are the \define{piecewise \( (\mathcal{F}, \mathcal{G}) \)-syndetic} sets.
\end{definition}
Shuungula, Zelenyuk, and Zelenyuk \cite[Section~2]{Shuungula:2009ty}  introduced a \emph{different} notion of relative piecewise syndetic sets: for a proper filter \( \mathcal{F} \) they defined \( A \subseteq S \) as \define{piecewise \( \mathcal{F} \)-syndetic} if and only if there exists \( q \in \overline{\mathcal{F}} \) such that \( A \in \mathsf{Syn}(\mathcal{F}, \mathsf{Thick}(\mathcal{F}, q)) \) (see \cite[Theorem~4.5]{Griffin:2024aa} for a proof of the equivalence between the definition found in \cite{Shuungula:2009ty} and the way we have presented it here).
Definition~\ref{definition:relative-piecewise-syndetic} was introduced in \cite{Christopherson:2022wr} to mainly ensure \( \mathsf{PS}(\mathcal{F}, \mathcal{G}) \) is a grill (by Propositions~\ref{proposition:relative-syndetic-thick} and \ref{proposition:grill}(b)), but it is not immediately clear if these two definitions of relative piecewise syndetic sets are equivalent.
If \( \mathcal{F} \) is a proper product filter, then  \cite[Corollary~4.5]{Christopherson:2022wr} shows \( A \in \mathsf{PS}(\mathcal{F}, \mathcal{F}) \) implies there exists \( q \in \overline{\mathcal{F}} \) with \( A \in \mathsf{Syn}(\mathcal{F}, \mathsf{Thick}(\mathcal{F}, q)) \) (that is, \( A \) is piecewise \( \mathcal{F} \)-syndetic in the sense of Shuungula, Zelenyuk, and Zelenyuk), while   \cite[Theorem~4.10]{Griffin:2024aa} proves the converse  (answering \cite[Question~4.6]{Christopherson:2022wr}).	

Our next goal is to obtain a characterization of relative piecewise syndetic sets analogous to Theorem~\ref{theorem:relative-syndetic-thick}.  
However, in order to obtain this goal and in contrast to the minimal assumptions placed on \( \mathcal{F} \) and  \( \mathcal{G} \) in Theorem~\ref{theorem:relative-syndetic-thick}, our proof will assume 
both \( \mathcal{F}, \mathcal{G} \) are proper filters and
they both satisfy some additional algebraic assumptions:
\( \mathcal{G} \) is a product filter and \( \mathcal{F} \subseteq \mathsf{Syn}(\mathcal{F}^*, \mathcal{G}) \).
The last assertion, by \cite[Theorem~3.11]{Christopherson:2022wr}, is equivalent to asserting \( \overline{\mathcal{F}} \cdot \overline{\mathcal{G}} \subseteq \overline{\mathcal{F}} \).
In this case, if \( \mathcal{G} \subseteq \mathcal{F} \), then this last assertion simply says \( \overline{\mathcal{F}} \) is a right ideal of \( \overline{\mathcal{G}} \).
But, in Example~\ref{example:derived-set-relative-syndetic} we note \( \overline{\mathcal{F}} \cdot \overline{\mathcal{G}} \subseteq \overline{\mathcal{F}} \) can also hold even when \( \mathcal{G} \not\subseteq \mathcal{F} \).

We start by noting elements of ``combinations'' (such as \( \mathsf{Syn}(\mathcal{F}, \mathsf{Thick}(\mathcal{G}, q)) \)) can be characterized in terms of derived sets:

\begin{proposition}
\label{proposition:derived-set-relative-syndetic}
	Let \( S \) be a semigroup, let \( \mathcal{F} \) be a proper stack, let \( \mathcal{G} \) be a proper filter, let \( q \in \beta S \), and let \( A \subseteq S \).
	The following are equivalent:

	\begin{itemize}
		\item[(a)]
		\( A'(q) \in \mathsf{Syn}(\mathcal{F}, \mathcal{G}) \).
				
		\item[(b)]
		\( A \in \mathsf{Thick}(\mathsf{Syn}(\mathcal{F}, \mathcal{G}), q) \).
				
		\item[(c)]
		\( A \in \mathsf{Syn}(\mathcal{F}, \mathsf{Thick}(\mathcal{G}, q)) \).
	\end{itemize}
\end{proposition}
\begin{proof}
	\textbf{(a) \( \Leftrightarrow \) (b):} By Definition~\ref{definition:product-of-collections} we have \( A'(q) \in \mathsf{Syn}(\mathcal{F}, \mathcal{G}) \) if and only if \( A \in \mathsf{Syn}(\mathcal{F}, \mathcal{G}) \cdot q \).
	Since \( \mathsf{Syn}(\mathcal{F}, \mathcal{G}) \) is a stack (Proposition~\ref{proposition:relative-syndetic-thick}(a)), by Corollary~\ref{corollary:relative-syndetic-thick}(b) (second assertion) we have \( \mathsf{Syn}(\mathcal{F}, \mathcal{G}) \cdot q  = \mathsf{Thick}(\mathsf{Syn}(\mathcal{F}, \mathcal{G}), q) \).
	
	\smallskip
	
	\textbf{(a) \( \Leftrightarrow \) (c):}
	By Corollary~\ref{corollary:relative-syndetic-thick}(a) we have \( A'(q) \in \mathsf{Syn}(\mathcal{F}, \mathcal{G}) \) if and only if for all \( p \in \overline{\mathcal{G}} \) we have \( A'(p \cdot q ) = \bigl(A'(q)\bigr)'(p) \in \mathcal{F}^* \), where the equality follows from Corollary~\ref{corollary:derived-set}(a, ii).
	Therefore it follows that \( A'(q) \in \mathsf{Syn}(\mathcal{F}, \mathcal{G}) \) if and only if for all \( p \in \overline{\mathcal{G}} \) we have \( A \in \mathcal{F}^* \cdot p \cdot q \).
	
	Now from Corollary~\ref{corollary:relative-syndetic-thick}(b) (second assertion) we have \( \mathsf{Thick}(\mathcal{G}, q) = \mathcal{G} \cdot q \) and, by Corollary~\ref{corollary:derived-set}(d) \( \mathcal{G} \cdot q \) is a proper filter.
	It follows that \(  \overline{\mathsf{Thick}(\mathcal{G},q)} = \overline{\mathcal{G} \cdot q} = \{ p \cdot q : p \in \overline{\mathcal{G}} \} \) (or, for instance, see \cite[Theorem~4.3]{Griffin:2024aa} for a proof of this fact).
	
	Hence \( A'(q) \in \mathsf{Syn}(\mathcal{F}, \mathcal{G}) \) if and only if for all \( r \in \overline{\mathsf{Thick}(\mathcal{G},q)} \) we have \( A \in \mathcal{F}^* \cdot r \).
	Therefore by Corollary~\ref{corollary:relative-syndetic-thick}(a) this means \( A'(q) \in \mathsf{Syn}(\mathcal{F}, \mathcal{G}) \) if and only if \( A \in \mathsf{Syn}(\mathcal{F}, \mathsf{Thick}(\mathcal{G}, q)) \).
\end{proof}

\begin{remark}
\label{remark:minimal-left-ideal}
    If \( \mathcal{F} \) is a proper product filter and \( q \in \overline{\mathcal{F}} \), then a result of Davenport \cite[Corollary~3.3]{Davenport:1990wq} states, in our notation, that \( \overline{\mathcal{F}} \cdot q \) is a minimal left ideal of \( \overline{\mathcal{F}} \) if and only if \( \mathsf{Syn}(\mathcal{F}, q) \subseteq \mathsf{Syn}(\mathcal{F}, \mathsf{Thick}(\mathcal{F}, q)) \) (compare with Theorem~\ref{theorem:maximal-elements}).
\end{remark}

Now we show each of above the algebraic assumptions on \( \mathcal{F} \) and \( \mathcal{G} \) implies certain derived sets of relative syndetic or thick sets are themselves relatively syndetic or thick:

\begin{lemma}
\label{lemma:derived-set-relative-syndetic}
	Let \( S \) be a semigroup, let \( \mathcal{F} \) be a proper stack, and let \( \mathcal{G} \) be a proper filter.
	\begin{itemize}
		\item[(a)]
			If \( \mathcal{G} \) is a product filter and \( B \in \mathsf{Syn}(\mathcal{F}, \mathcal{G}) \), then for all \( q \in \overline{\mathcal{G}} \) we have \( B'(q) \in \mathsf{Syn}(\mathcal{F}, \mathcal{G}) \).
			
		\item[(b)]
			If \( \mathcal{F} \subseteq \mathsf{Syn}(\mathcal{F}^*, \mathcal{G}) \) and \( C \in \mathsf{Thick}(\mathcal{F}, \mathcal{G}) \), then there exists \( q \in \overline{\mathcal{G}} \) such that \( C'(q) \in \mathsf{Syn}(\mathcal{F}^*, \mathcal{G}) \).
			In particular, \( C'(q) \in \mathsf{Thick}(\mathcal{F}, \mathcal{G}) \).
	\end{itemize}
\end{lemma}
\begin{proof}
	\textbf{(a):}
	From Corollary~\ref{corollary:relative-syndetic-thick}(a) we have for all \( p \in \overline{\mathcal{G}} \) that \( B'(p) \in \mathcal{F}^* \).
	Fix (any) \( q \in \overline{\mathcal{G}} \).
	Then for all \( p \in \overline{\mathcal{G}} \) we have \( p \cdot q \in \overline{\mathcal{G}} \) and \( B'(p \cdot q) \in \mathcal{F}^* \).
	Therefore from Corollary~\ref{corollary:derived-set} it follows that \( B'(q) \in \mathcal{F}^* \cdot p \) for all \( p \in \overline{\mathcal{G}} \), that is, from Corollary~\ref{corollary:relative-syndetic-thick}(a) \( B'(q) \in \mathsf{Syn}(\mathcal{F}, \mathcal{G}) \).
	\smallskip
	
	\textbf{(b):}
	The proof of this statement is similar to the proof of statement (a).
	From Corollary~\ref{corollary:relative-syndetic-thick}(b) pick \( q \in \overline{\mathcal{G}} \) with \( C \in \mathcal{F} \cdot q \).
	By the algebraic assumption on \( \mathcal{F} \) and Corollary~\ref{corollary:relative-syndetic-thick}(a) we have \( \mathcal{F} \subseteq \mathcal{F} \cdot p \) for all \( p \in \overline{\mathcal{G}} \).
	Hence, from Corollary~\ref{corollary:derived-set}(b), for all \( p \in \overline{\mathcal{G}} \) we have \( C \in \mathcal{F} \cdot q \subseteq \mathcal{F} \cdot p \cdot q \).
	So, for all \( p \in \overline{\mathcal{G}} \) we have \( C'(q) \in \mathcal{F} \cdot p \) and by Corollary~\ref{corollary:relative-syndetic-thick}(a) we have \( C'(q) \in \mathsf{Syn}(\mathcal{F}^*, \mathcal{G}) \).
	The `in particular' statement also follows from Corollary~\ref{corollary:relative-syndetic-thick}, which implies \( \mathsf{Syn}(\mathcal{F}^*, \mathcal{G}) \subseteq  \mathsf{Thick}(\mathcal{F}, \mathcal{G}) \).
\end{proof}

\begin{example}
\label{example:derived-set-relative-syndetic}
    Recall the proper filters \( \mathcal{H} \) and \( \mathcal{D} \) from Example~\ref{example:filters}.
    By \cite[Theorem~7.12]{Hindman:1980aa} \( \overline{\mathcal{D}} = \{ p \in \beta \mathbb{N} : (\forall \, A \in p) \; \mathsf{bd}^\star(A) > 0 \} \) is a closed two-sided ideal of \( (\beta \mathbb{N}, +) \).
    (See also the recent result of Glasscock, Hindman, and Strauss \cite[Theorem~2.8]{Glasscock:2025aa} for an analogous statement for left amenable semigroups.)
    Hence \( \mathcal{D} \subseteq \mathsf{Syn}(\mathcal{D}^*, \mathcal{G}) \) for all proper filters \( \mathcal{G} \).
    Adams \cite[Theorem~2.21]{Adams:2007aa} gives an explicit construction of a moderately growing sequence \( \langle x_n \rangle_{n = 1}^\infty \) such that \( \overline{\mathcal{H}} \cap c\ell(K(\beta\mathbb{N})) = \emptyset \) and \( \emptyset \ne  \overline{\mathcal{D}} \cap \overline{\mathcal{H}} \) (compare with Example~\ref{example:product-filters}(b)).
    Therefore by Example~\ref{example:product-filters}(a) and Proposition~\ref{proposition:product-filters} we have \( \mathcal{D} \sqcap \mathcal{H} \) is a proper product filter.
    Since \( \mathcal{D} \subseteq \mathsf{Syn} \), we have \( \mathcal{H} \not\subseteq \mathcal{D} \) and so \( \mathcal{D} \sqcap \mathcal{H}  \not\subseteq \mathcal{D} \) but \( \mathcal{D} \subseteq \mathsf{Syn}(\mathcal{D}^*, \mathcal{D} \sqcap \mathcal{H}) \).

\end{example}

In the proof for Lemma~\ref{lemma:derived-set-relative-syndetic}(a), since \( \overline{\mathcal{G}} \) is a closed subsemigroup, we have a great deal of freedom in picking an ultrafilter \( q \) witnessing the relative syndeticity of \( B'(q) \).
For instance \( q \) could be a (minimal) idempotent in \( \overline{\mathcal{G}} \). 
This idea appears (in a more complicated form) several times in the proof of our main result characterizing relative piecewise syndetic sets: 

\begin{theorem}
\label{theorem:relative-piecewise-syndetic}
	Let \( \mathcal{F}, \mathcal{G} \) both be proper filters on a semigroup \( S \) with \( \mathcal{F} \subseteq \mathsf{Syn}(\mathcal{F}^*, \mathcal{G}) \) and \( \mathcal{G} \) a product filter.
	For \( A \subseteq S \) the following statements are equivalent:
	\begin{itemize}
		\item[(a)]
			\( A \in \mathsf{PS}(\mathcal{F}, \mathcal{G}) \).
			
		\item[(b)] 
			There exists \( q \in \overline{\mathcal{G}} \) such that \( A'(q) \in \mathsf{Syn}(\mathcal{F}, \mathcal{G}) \).
			
		\item[(c)]
			There exists \( q \in K(\overline{\mathcal{G}}) \) such that \( A'(q) \in \mathsf{Syn}(\mathcal{F}, \mathcal{G}) \).
			
		\item[(d)]
			There exists \( e \in E\bigl(K(\overline{\mathcal{G}})\bigr) \) such that \( A'(e) \in \mathsf{Syn}(\mathcal{F}, \mathcal{G}) \).
			
		\item[(e)]
			There exists \( e \in E(\overline{\mathcal{G}}) \) such that \( A'(e) \in \mathsf{Syn}(\mathcal{F}, \mathcal{G}) \).
	\end{itemize}
\end{theorem}

\begin{proof}

    Since \( \overline{\mathcal{G}} \) is a compact right topological semigroup its smallest ideal \( K(\overline{\mathcal{G}}) \) exists, \( K(\overline{\mathcal{G}}) \) is a union of all minimal left ideals of \( \overline{\mathcal{G}} \), and every minimal left ideal \( L \) of \( \overline{\mathcal{G}} \) contains an idempotent.

	\smallskip
	
	\textbf{(a) \( \Rightarrow \) (b):}
	Pick \( B \in \mathsf{Syn}(\mathcal{F}, \mathcal{G}) \) and \( C \in \mathsf{Thick}(\mathcal{F}, \mathcal{G}) \) as guaranteed for \( A \).
	By Lemma~\ref{lemma:derived-set-relative-syndetic}(b) pick \( q \in \overline{\mathcal{G}} \) such that \( C'(q) \in \mathsf{Syn}(\mathcal{F}^*, \mathcal{G}) \).
	Then Lemma~\ref{lemma:derived-set-relative-syndetic}(a) implies \( B'(q) \in \mathsf{Syn}(\mathcal{F}, \mathcal{G}) \).
	
	From Proposition~\ref{proposition:derived-set}(a, iii) we have \( A'(q) = (B \cap C)'(q) = B'(q) \cap C'(q) \) and we show \( B'(q) \cap C'(q) \in \mathsf{Syn}(\mathcal{F}, \mathcal{G}) \).
	To that end let \( p \in \overline{\mathcal{G}} \).
	By Corollary~\ref{corollary:relative-syndetic-thick}(a) we have \( B'(q) \in \mathcal{F}^* \cdot p \) and \( C'(q) \in \mathcal{F} \cdot p \).
	We have \( \mathcal{F} \subseteq \mathcal{F}^* \) (by Proposition~\ref{proposition:filter-grill}(b), since \( \mathcal{F} \) is a proper filter) and Corollary~\ref{corollary:derived-set}(c) yields \( \mathcal{F} \cdot p \subseteq \mathcal{F}^* \cdot p \).
	Hence it follows that \( B'(q) \cap C'(q) \in \mathcal{F}^* \cdot p \), and from Corollary~\ref{corollary:relative-syndetic-thick}(a) it follows \( B'(q) \cap C'(q) \in \mathsf{Syn}(\mathcal{F}, \mathcal{G}) \).

	\smallskip

	\textbf{(b) \( \Rightarrow \) (c):}
	Pick \( q \in \overline{\mathcal{G}} \) as guaranteed for \( A \) and fix \( p \in K(\overline{\mathcal{G}}) \).
	Then, by Corollary~\ref{corollary:derived-set}(a, ii), for all \( u \in \overline{\mathcal{G}} \) we have
	\[
		A'(q) \in \mathcal{F}^* \cdot u \cdot p \iff A'(p \cdot q) \in \mathcal{F}^* \cdot u,
	\]
	Because \( K(\overline{\mathcal{G}}) \) is an ideal of \( \overline{\mathcal{G}} \) we have both \( u \cdot p \) and \( p \cdot q \) are elements of \( K(\overline{\mathcal{G}}) \subseteq \overline{\mathcal{G}} \).
	Therefore by Corollary~\ref{corollary:relative-syndetic-thick}(a) applied to \( A'(q) \in \mathsf{Syn}(\mathcal{F}, \mathcal{G}) \) we have \( A'(q) \in \mathcal{F}^* \cdot u \cdot p \) for all \( u \in \overline{\mathcal{G}} \).
	Hence, from Corollary~\ref{corollary:relative-syndetic-thick}(a) and the above equivalence we have \( A'(p \cdot q) \in \mathsf{Syn}(\mathcal{F}, \mathcal{G}) \).
		
	\smallskip
	
	\textbf{(c) \( \Rightarrow \) (d):}
	Pick \( q \in K(\overline{\mathcal{G}}) \) as guaranteed for \( A \).
	Let \( L \) be a minimal left ideal of \( \overline{\mathcal{G}} \) with \( q \in L \) and let \( e \) be an idempotent in \( L \).
	Then \( \overline{\mathcal{G}} \cdot q \) is a left ideal of \( \overline{\mathcal{G}} \) contained in \( L \).
	Therefore by minimality of \( L \) we have \( \overline{\mathcal{G}} \cdot q = L \).

	We show \( A'(e) \in \mathsf{Syn}(\mathcal{F}, \mathcal{G}) \).
	To that end let \( p \in \overline{\mathcal{G}} \).
	Then \( p \cdot e \in L = \overline{\mathcal{G}} \cdot q \) implies there exists \( u \in \overline{\mathcal{G}} \) such that \( p \cdot e = u \cdot q \).
	We then have
	\begin{align*}
		A'(q) \in \mathsf{Syn}(\mathcal{F}, \mathcal{G}) &\implies A'(q) \in \mathcal{F}^* \cdot u &\text{(Corollary~\ref{corollary:relative-syndetic-thick}(a))}\\
		&\iff A \in \mathcal{F}^* \cdot u \cdot q = \mathcal{F}^* \cdot p \cdot e \\
		&\iff A'(e) \in \mathcal{F}^* \cdot p.
	\end{align*}
	Therefore from Corollary~\ref{corollary:relative-syndetic-thick}(a) it follows that \( A'(e) \in \mathsf{Syn}(\mathcal{F}, \mathcal{G}) \).

	\smallskip
	
	\textbf{(d) \( \Rightarrow \) (e):}
	This assertion is immediate.
	
	\smallskip

	\textbf{(e) \( \Rightarrow \) (a):}
	Pick \( e \in E(\overline{\mathcal{G}}) \) as guaranteed for \( A \) and note we have \( A \cup A'(e) \in \mathsf{Syn}(\mathcal{F}, \mathcal{G}) \).
	By Proposition~\ref{proposition:derived-set}(a, iii) we have
	\[
		A = \bigl( A \cup A'(e) \bigr) \cap \bigl( A \cup (S\setminus A)'(e) \bigr).
	\]
	So it suffices to show \( A \cup (S\setminus A)'(e) \in \mathsf{Thick}(\mathcal{F}, \mathcal{G}) \).
	To that end note 
		\begin{align*}
		\bigl( A \cup (S\setminus A)'(e) \bigr)'(e) &= A'(e) \cup \bigl( (S \setminus A)'(e) \bigr)'(e) &\text{(Proposition~\ref{proposition:derived-set}(a, iii))} \\
		&= A'(e) \cup (S \setminus A)'(e \cdot e) &\text{(Corollary~\ref{corollary:derived-set}(a, ii))} \\
		&= A'(e) \cup (S \setminus A)'(e)  &\text{(\( e = e \cdot e \))} \\
		&= \bigl(A \cup (S\setminus A)\bigr)'(e) &\text{(Proposition~\ref{proposition:derived-set}(a, iii))}\\
		&= S'(e) = S \in \mathcal{F}.
	\end{align*}
	Therefore from Corollary~\ref{corollary:relative-syndetic-thick}(b) we have \( A \cup (S \setminus A)'(e) \in \mathsf{Thick}(\mathcal{F}, \mathcal{G}) \).
	(Note the above, again combined with Corollary~\ref{corollary:relative-syndetic-thick}(b), shows something stronger: \( A \cup (S \setminus A)'(e) \) is thick.)
\end{proof}

Therefore by Proposition~\ref{proposition:derived-set-relative-syndetic} and Theorem~\ref{theorem:relative-piecewise-syndetic} if \( \mathcal{F}, \mathcal{G} \) are proper filters with \( \mathcal{F} \subseteq \mathsf{Syn}(\mathcal{F}^*, \mathcal{G}) \) and \( \mathcal{G} \) is a product filter, then we have, for instance,
\[
    \mathsf{PS}(\mathcal{F}, \mathcal{G}) = \bigcup_{q \in \overline{\mathcal{G}}} \mathsf{Syn}(\mathcal{F}, \mathsf{Thick}(\mathcal{G}, q)) = 
       \bigcup_{e \in E(K(\overline{\mathcal{G}}))} \mathsf{Syn}(\mathcal{F}, \mathsf{Thick}(\mathcal{G}, e)).
\]
Further we have, under the appropriate assumptions on proper filters \( \mathcal{F} \) and \( \mathcal{G} \), that a generalization of Shuungula, Zelenyuk, and Zelenyuk's original definition of piecewise \( \mathcal{F} \)-syndetic \cite[Section~2]{Shuungula:2009ty} also holds:
\begin{corollary}
\label{corollary:relative-piecewise-syndetic}
    Let \( \mathcal{F}, \mathcal{G} \) both be proper filters on a semigroup \( S \) with \( \mathcal{F} \subseteq \mathsf{Syn}(\mathcal{F}^*, \mathcal{G}) \) and \( \mathcal{G} \) a product filter.
	For \( A \subseteq S \) the following statements are equivalent:
    \begin{itemize}
        \item[(a)]
            \( A \in \mathsf{PS}(\mathcal{F}, \mathcal{G}) \).

        \item[(b)]
            There exist \( H \in \bigtimes_{B \in \mathcal{F}} \mathcal{P}_f(B) \) and \( W \in \bigtimes_{B \in \mathcal{F}} \mathcal{G} \) such that
            \[
                \bigl\{ \textstyle w^{-1}\bigl( \bigcup_{h \in H(B)} h^{-1}A \bigr) : B \in \mathcal{F} \text{ and } w \in W(B) \bigr\} \cup \mathcal{G}
            \]
            has the finite intersection property (f.i.p.).
    \end{itemize}
\end{corollary}
\begin{proof}
    \textbf{(a) \( \Rightarrow \) (b):}
    By Theorem~\ref{theorem:relative-piecewise-syndetic} pick \( q \in \overline{\mathcal{G}} \) with \( A'(q) \in \mathsf{Syn}(\mathcal{F}, \mathcal{G}) \).
    Then by Definition~\ref{definition:relative-syndetic-thick}(a) there exists \( H \in \bigtimes_{B \in \mathcal{F}} \mathcal{P}_f(B) \) such that \( \bigcup_{h \in H(B)} h^{-1}A'(q) \in \mathcal{G} \) for each \( B \in \mathcal{F} \).
    By Proposition~\ref{proposition:derived-set}(a, iii) and (d) we have each \( \bigcup_{h \in H(B)} h^{-1}A'(q) = \bigl( \bigcup_{h \in H(B)} h^{-1}A \bigr)'(q) \).
    Put \( W(B) :=  \bigl( \bigcup_{h \in H(B)} h^{-1}A \bigr)'(q) \) for each \( B \in \mathcal{F} \).

    To see that f.i.p.~holds let \( B_1, B_2, \ldots, B_n \in \mathcal{F} \),  for each \( i \in \{1, 2, \ldots, n\} \) let \( W_i \in \mathcal{P}_f(W(B_i)) \), and let \( C \in \mathcal{G} \).
    Then we have \( C \in q \) and for each \( i \in \{1, 2, \ldots, n \} \) we have \( w \in W_i \) implies \( w^{-1}\bigl(\bigcup_{h \in H(B_i)} h^{-1}A \bigr) \in q \).
    Therefore
    \[
        C \cap \bigcap\bigl\{ \textstyle w^{-1}\bigl(\bigcup_{h \in H(B_i)} h^{-1}A) : 
            i \in \{1, 2, \ldots, n\} \text{ and } 
            w \in W_i \bigr\} \in q
    \]

    \smallskip

    \textbf{(b) \( \Rightarrow \) (a):}
    Pick \( H \in \bigtimes_{B \in \mathcal{F}} \mathcal{P}_f(B) \) and \( W \in \bigtimes_{B \in \mathcal{F}} \mathcal{G} \) as guaranteed for \( A \).
    By assumption let \( q \in \beta S \) with  \( \bigl\{ \textstyle w^{-1}\bigl( \bigcup_{h \in H(B)} h^{-1}A \bigr) : B \in \mathcal{F} \text{ and } w \in W(B) \bigr\} \cup \mathcal{G} \subseteq q \).
    Since \( \mathcal{G} \subseteq q \) we have \( q \in \overline{\mathcal{G}} \) and note \( W(B) \subseteq \bigl( \bigcup_{h \in H(B)} h^{-1}A \bigr)'(q) \) for every \( B \in \mathcal{F} \).
    Hence, by Proposition~\ref{proposition:derived-set}(a, iii) and (d) we have \( \bigcup_{h \in H(B)} h^{-1}A'(q) \in \mathcal{G} \) for each \( B \in \mathcal{G} \), that is, \( A'(q) \in \mathsf{Syn}(\mathcal{F}, \mathcal{G}) \).
    Therefore by Theorem~\ref{theorem:relative-piecewise-syndetic} we have \( A \in \mathsf{PS}(\mathcal{F}, \mathcal{G}) \).
\end{proof}

\begin{example}
\label{example:relative-piecewise-syndetic}
    Recall the proper filters \( \mathcal{C} \), \( \mathcal{H} \), and \( \mathcal{D} \) from Example~\ref{example:filters}.
    As in Example~\ref{example:relative-syndetic-thick}(b) suppose we have picked a sequence \( \langle x_n \rangle_{n = 1}^\infty \) such that each \( \mathrm{FS}(\langle x_n \rangle_{n = m}^\infty) \) is \emph{not} (piecewise) syndetic.

    \begin{itemize}
        \item[(a)]   
            By Proposition~\ref{proposition:relative-syndetic-thick}, 
                for every collections \( \mathcal{F}, \mathcal{G} \) 
                    if \( \mathsf{Thick}(\mathcal{F}, \mathcal{G}) = \mathsf{Thick} \) 
                    (equivalently:  \( \mathsf{Syn}(\mathcal{F}, \mathcal{G}) = \mathsf{Syn} \)), 
                    then  \( \mathsf{PS}(\mathcal{F}, \mathcal{G}) = \mathsf{PS} \).
            In particular, as noted in \cite[Example~3.6]{Christopherson:2022wr} we have \( \mathsf{Thick}(\mathcal{C}, \mathcal{C}) = \mathsf{Thick} \) in \( (\mathbb{N}, +) \) and so
            therefore \( \mathsf{PS}(\mathcal{C}, \mathcal{C}) = \mathsf{PS} \).
            Even though we have this equivalence between ``relative'' and ``absolute'' piecewise syndetic notions, it is interesting to note, say from Corollary~\ref{corollary:relative-piecewise-syndetic}, \( A \in \mathsf{PS}(\mathcal{C}, \mathcal{C}) \) if and only if for all \( k \in \mathbb{N} \) there exist \( \ell \in \mathbb{N} \) with \( \ell > k \) and \( m \in \mathbb{N} \) such that for all \( x \in \mathbb{N} \) with \( x > m \) there exists \( h_x \in [k, \ell] \) such that \( \{ -x + (-h_x + A) : x \in \mathbb{N} \text{ and } x > m \} \) has the infinite finite intersection property, that is, for finitely many positive integers \( x_1, x_2, \ldots, x_n \) with each \( x_i > m \) we have \( \bigcap_{i = 1}^n -x_i + (-h_{x_i} +A) \) is infinite.

        \item[(b)]
            Similarly, one can show \( \mathsf{Thick}(\mathcal{D}, \mathcal{D}) = \mathsf{Thick} \) and so \( \mathsf{PS}(\mathcal{D}, \mathcal{D}) = \mathsf{PS} \).
            
        \item[(c)]
            For every collection \( \mathcal{F}, \mathcal{G} \) we have \( \mathsf{Syn}(\mathcal{F}, \mathcal{G}) \subseteq \mathsf{PS}(\mathcal{F}, \mathcal{G}) \).
            Hence, as noted in Example~\ref{example:relative-syndetic-thick}(b) we have \( \mathsf{PS}(\mathcal{H}, \mathcal{H}) \not\subseteq \mathsf{PS} \).

    \end{itemize}
\end{example}

We note that in the above proof of Theorem~\ref{theorem:relative-piecewise-syndetic} we \emph{only used the assumption} that \( \mathcal{F} \) is a proper filter with \( \mathcal{F} \subseteq \mathsf{Syn}(\mathcal{F}^*, \mathcal{G}) \) in showing statement (a) implies statement (b), where we invoked Lemma~\ref{lemma:derived-set-relative-syndetic}(b).
However, the next result appears to use both assumptions on \( \mathcal{F} \) and \( \mathcal{G} \) in a rather strong way.
This will enable us to prove (in Theorem~\ref{theorem:relative-kernel}) a result analogous to the characterization of ultrafilters in the smallest ideal \( K(\beta S) \) \cite[Theorem~4.39]{Hindman:2012tq}.
If \( T \) is a closed subsemigroup of \( \beta S \), then Davenport \cite[Theorem~3.4]{Davenport:1990wq} characterized ultrafilters in the smallest ideal \( K(T) \) and, independently, a similar characterization was obtained in \cite[Theorem~2.2]{Shuungula:2009ty}.
Baglini, Patra, and Shaikh \cite[Theorem~2.9]{Luperi-Baglini:2023aa} extended this latter characterization into a form more closely matching \cite[Theorem~4.39]{Hindman:2012tq}.

\begin{definition}
\label{definition:relative-kernel}
	Let \( \mathcal{F}, \mathcal{G} \) both be proper filters on a semigroup \( S \) with \( \mathcal{G} \) a product filter.
	Put
	\[
		K(\mathcal{F}, \mathcal{G}) := \overline{\mathcal{F}} \cdot K(\overline{\mathcal{G}}).
	\]
\end{definition}

\begin{theorem}
\label{theorem:relative-kernel}	
	Let \( \mathcal{F}, \mathcal{G} \) both be proper filters on a semigroup \( S \) with \( \mathcal{F} \subseteq \mathsf{Syn}(\mathcal{F}^*, \mathcal{G}) \) and \( \mathcal{G} \) a product filter.
	Let \( p \in \beta S \).
	The following statements are equivalent:
	\begin{itemize}
		\item[(a)]
			\( p \in K(\mathcal{F}, \mathcal{G}) \).
			
		\item[(b)]
			There exists \( e \in E(K(\overline{\mathcal{G}})) \) such that \( p \in \overline{\mathcal{F}} \cdot e \).
			
		\item[(c)]
			There exists \( e \in E(K(\overline{\mathcal{G}})) \) such that for all \( q \in \overline{\mathcal{G}} \) we have \( p \in \overline{\mathcal{F}} \cdot q \cdot e \).
			
		\item[(d)]
			There exists \( e \in E(K(\overline{\mathcal{G}})) \) such that for all \( A \in p \) we have \( A'(e) \in \mathsf{Syn}(\mathcal{F}, \mathcal{G}) \).
	\end{itemize}
\end{theorem}
\begin{proof}
	\textbf{(a) \( \Rightarrow \) (b):}
	Pick \( r \in \overline{\mathcal{F}} \) and \( q \in K(\overline{\mathcal{G}}) \) such that \( p = r \cdot q \).
	Let \( L \) be a minimal left ideal of \( \overline{\mathcal{G}} \) with \( q \in L \) and let \( e \in E(L) \).
	By minimality we have \( L = \overline{\mathcal{G}} \cdot e \).
	Therefore we can pick \( u \in \overline{\mathcal{G}} \) with \( q = u \cdot e \), and so \( p = r \cdot q = r \cdot u \cdot e \)
	Note, by assumption on \( \mathcal{F} \), it follows that \( r \cdot u \in \overline{\mathcal{F}} \).
	Hence \( p \in \overline{\mathcal{F}} \cdot e \).
	
	\smallskip
	
	\textbf{(b) \( \Rightarrow \) (c):}
	Pick \( e \in E(K(\overline{\mathcal{G}})) \) as guaranteed and let \( u \in \overline{\mathcal{F}} \) with \( p = u \cdot e \).
	By \cite[Theorem~2.9]{Luperi-Baglini:2023aa} since \( e \in K(\overline{\mathcal{G}}) \) for all \( q \in \overline{\mathcal{G}} \) we have \( e \in \overline{\mathcal{G}} \cdot q \cdot e \).
	Then for all \( q \in \overline{\mathcal{G}} \) we have \( p = u \cdot e \in u \cdot \overline{\mathcal{G}} \cdot q \cdot e \subseteq \overline{\mathcal{F}} \cdot q \cdot e \), where \( u \cdot \overline{\mathcal{G}} \subseteq \overline{\mathcal{F}} \) follows from our assumption on \( \mathcal{F} \).
	
	\smallskip
	
	\textbf{(c) \( \Rightarrow \) (d):}
	Pick \( e \in E(K(\overline{\mathcal{G}})) \) as guaranteed.
	By assumption for all \( q \in \overline{\mathcal{G}} \) and \( A \in p \) we have \( p \in \overline{\mathcal{F}} \cdot q \cdot e \) implies \( A \in \mathcal{F}^* \cdot q \cdot e \), that is, \( A'(e) \in  \mathcal{F}^* \cdot q \).
	Hence from Corollary~\ref{corollary:relative-syndetic-thick}(a) we have \( A'(e) \in \mathsf{Syn}(\mathcal{F}, \mathcal{G}) \).
	
	\smallskip
	\textbf{(d) \( \Rightarrow \) (a):}
	Pick \( e \in E(K(\overline{\mathcal{G}})) \) as guaranteed and fix \( q \in \overline{\mathcal{G}} \).
	Then from Corollary~\ref{corollary:relative-syndetic-thick}(a) it follows \( p \subseteq \mathcal{F}^* \cdot q \cdot e \).
	Since, by Corollary~\ref{corollary:derived-set}(d, iii) \( \mathcal{F}^* \cdot q \cdot e \) is a grill we have \( p \subseteq \mathcal{F}^* \cdot q \cdot e \) implies \( p \in \overline{\mathcal{F}} \cdot q \cdot e \) (for instance, see \cite[Lemma~4.2]{Griffin:2024aa}).
	Because \( K(\overline{\mathcal{G}}) \) is an ideal of \( \overline{\mathcal{G}} \) we have \( q \cdot e \in K(\overline{\mathcal{G}}) \), and so \( p \in \overline{\mathcal{F}} \cdot K(\overline{\mathcal{G}}) \).
\end{proof}

One of the main results of Shuungula, Zelenyuk, and Zelenyuk \cite[Theorem~2.3 and Corollary~2.4]{Shuungula:2009ty} proves their notion of piecewise \( \mathcal{F} \)-syndetic can be characterized in terms of the smallest ideal \( K(\overline{\mathcal{F}}) \).
An analogous result holds for piecewise \( (\mathcal{F}, \mathcal{G}) \)-syndetic and \( K(\mathcal{F}, \mathcal{G}) \).

\begin{corollary}
\label{corollary:relative-kernel}	
	Let \( \mathcal{F}, \mathcal{G} \) both be proper filters on a semigroup \( S \) with \( \mathcal{F} \subseteq \mathsf{Syn}(\mathcal{F}^*, \mathcal{G}) \) and \( \mathcal{G} \) a product filter.
	\begin{itemize}
		\item[(a)]
			For all \( A \subseteq S \) we have \( A \in \mathsf{PS}(\mathcal{F}, \mathcal{G}) \) if and only if \( \overline{A} \cap K(\mathcal{F}, \mathcal{G}) \ne \emptyset \).
			
		\item[(b)]
			\( c\ell\bigl( K(\mathcal{F}, \mathcal{G}) \bigr) = \{ p \in \beta S : p \subseteq \mathsf{PS}(\mathcal{F}, \mathcal{G}) \} \).

        \item[(c)]
            \(K(\mathcal{G}, \mathcal{G}) = K(\overline{\mathcal{G}}) \).
	\end{itemize}
\end{corollary}
\begin{proof}
	\textbf{(a):}
	By \cite[Lemma~3.9]{Christopherson:2022wr}, we have \( \mathsf{Syn}(\mathcal{F}, \mathcal{G}) = \{ B \subseteq S : \text{for all } p \in \overline{\mathcal{G}} \text{ we have } \overline{B} \cap \overline{\mathcal{F}} \cdot p \ne \emptyset \} \).
	
	\smallskip

	\textbf{(\( \Rightarrow \)):}
	Let \( A \in \mathsf{PS}(\mathcal{F}, \mathcal{G}) \) and, by Theorem~\ref{theorem:relative-piecewise-syndetic} pick \( q \in K(\overline{\mathcal{G}}) \) such that \( A'(q) \in \mathsf{Syn}(\mathcal{F}, \mathcal{G}) \).
	Hence, by the above, for all \( p \in \overline{\mathcal{G}} \) we have \( \overline{A'(q)} \cap \overline{\mathcal{F}} \cdot p \ne \emptyset \), and by Corollary~\ref{corollary:derived-set}(a, i) it follows that \( \overline{A} \cap \overline{\mathcal{F}} \cdot p \cdot q \ne \emptyset \).
	Since \( K(\overline{\mathcal{G}}) \) is an ideal of \( \overline{\mathcal{G}} \) we have \( p \cdot q \in K(\overline{\mathcal{G}}) \) and so \( \overline{A} \cap K(\mathcal{F}, \mathcal{G}) \ne \emptyset \).
	
	\textbf{(\( \Leftarrow \)):}
	Pick \( p \in \overline{A} \cap K(\mathcal{F}, \mathcal{G}) \) as guaranteed.
	Then, by Theorem~\ref{theorem:relative-kernel}, there exists \( e \in E(K(\overline{\mathcal{G}})) \) such that \( A'(e) \in \mathsf{Syn}(\mathcal{F}, \mathcal{G}) \).
	Therefore by Theorem~\ref{theorem:relative-piecewise-syndetic} we have \( A \in \mathsf{PS}(\mathcal{F}, \mathcal{G}) \).
	
	\smallskip

	\textbf{(b):}
	This follows from statement (a).
	Let \( p \in \beta S \). 
	If \( p \in c\ell\bigl( K(\mathcal{F}, \mathcal{G}) \bigr) \), then for all \( A \in p \) we have \( \overline{A} \cap K(\mathcal{F}, \mathcal{G}) \ne \emptyset \), and so by statement (a) we have \( A \in \mathsf{PS}(\mathcal{F}, \mathcal{G}) \).
	If \( p \subseteq \mathsf{PS}(\mathcal{F}, \mathcal{G}) \), then for all \( A \in p \), by statement (a), we have \( \overline{A} \cap K(\mathcal{F}, \mathcal{G}) \ne \emptyset \).
	Hence \( p \in c\ell\bigl( K(\mathcal{F}, \mathcal{G}) \bigr) \).

    \smallskip
    
    \textbf{(c):}
    Since \( K(\overline{\mathcal{G}}) \) is a two-sided ideal of \( \overline{\mathcal{G}} \) we have \( K(\mathcal{G}, \mathcal{G}) \subseteq  K(\overline{\mathcal{G}}) \).
    For the reverse inclusion note \( K(\mathcal{G}, \mathcal{G}) \) is a  two-sided ideal of \( \overline{\mathcal{G}} \) and hence contains the smallest ideal \( K(\overline{\mathcal{G}}) \).
    
\end{proof}

\section{Collectionwise relative notions of size and relative central sets}
\label{section:collectionwise}

This section applies the results in Section~\ref{section:relative-notions} to give (in our view) an easier characterization of `collectionwise' (relative) notions (Corollary~\ref{corollary:collectionwise-relative-piecewise-syndetic}) in terms of derived sets along an ultrafilter (Definition~\ref{definition:collectionwise-relative-piecewise-syndetic}).
Under the assumptions of Theorem~\ref{theorem:relative-kernel}, such sets are precisely those collections contained in an ultrafilter of \( K(\mathcal{F}, \mathcal{G}) \) (Theorem~\ref{theorem:collectionwise-relative-piecewise-syndetic}).
We end this section with a question asking if our defined relative notion of a central set (Definition~\ref{definition:relative-central}) satisfies the Ramsey property (Question~\ref{question:partition-regular}): the answer is ``yes'' under the assumptions of Theorem~\ref{theorem:relative-kernel} and one additional assumption (Corollary~\ref{corollary:relative-central-is-partition-regular}) but the question is open for weaker assumptions on the underlying filters.

Hindman and Lisan \cite{Hindman:1994aa} (see \cite[Theorem~14.21]{Hindman:2012tq} for the appropriate left-right switch version) characterized those collections \( \mathcal{A} \subseteq \mathcal{P}(S) \) which are contained in an ultrafilter of the smallest ideal \( K(\beta S) \): there exists \( p \in K(\beta S) \) with \( \mathcal{A} \subseteq p \) if and only if \( \mathcal{A} \) is collectionwise piecewise syndetic.
Hindman,  Maleki, and Strauss \cite{Hindman:1996aa} (see \cite[Theorem~14.25]{Hindman:2012tq}) also used this notion of collectionwise piecewise syndetic to characterize which collections are contained in idempotents of the smallest ideal \( K(\beta S) \).
The definition of collectionwise piecewise syndetic \cite[Definition~14.19 and Exercise 14.5.1]{Hindman:2012tq} is complicated:
\begin{definition}
\label{definition:collectionwise-piecewise-syndetic}
    Let \( S \) be a semigroup and let \( \mathcal{A} \) a collection on \( S \).
    We call \( \mathcal{A} \) \define{collectionwise piecewise syndetic} if and only if there exists a function \(  H \colon \mathcal{P}_f(\mathcal{A}) \to \mathcal{P}_f(S) \) such that
    \[ \textstyle
        \bigl\{ w^{-1}\bigl( \bigcup_{h \in H(\mathcal{B})} h^{-1}(\bigcap\mathcal{B}) \bigr) : \mathcal{B} \in \mathcal{P}_f(\mathcal{A}) \text{ and } w \in S  \bigr\}
    \]
    has the finite intersection property (f.i.p.).

\end{definition}

Recently, Luperi Baglini, Patra, and Shaikh generalized both of the above results to show for a proper product filter \( \mathcal{F} \) there exists \( p \in K(\overline{\mathcal{F}}) \) with \( \mathcal{A} \subseteq p \) if and only if \( \mathcal{A} \) is collectionwise piecewise \( \mathcal{F} \)-syndetic \cite[Theorem~5.2]{Luperi-Baglini:2023aa}, and they also characterized which collections are members of idempotents in \( K(\overline{\mathcal{F}}) \) \cite[Theorem~5.8]{Luperi-Baglini:2023aa}.
Similarly, their definition \cite[Definition~5.1]{Luperi-Baglini:2023aa} is complicated to state, but in particular we have \( A \) is piecewise (\( \mathcal{F} \)-)syndetic if and only if \( \{A\} \) is collectionwise piecewise  (\( \mathcal{F} \)-)syndetic. %

We now introduce a generalization (see Corollary~\ref{corollary:collectionwise-relative-piecewise-syndetic}) of these two previous notions of collectionwise (relative) piecewise syndetic sets, and simplify the definition via derived along an ultrafilter:
\begin{definition}
\label{definition:collectionwise-relative-piecewise-syndetic}
	Let \( S \) be a semigroup, let \( \mathcal{F}, \mathcal{G} \) both be proper filters on \( S \), and let \( \mathcal{A} \) be a collection on \(S \).
	We call \( \mathcal{A} \) \define{collectionwise piecewise \( (\mathcal{F}, \mathcal{G}) \)-syndetic} if and only if there exists \( q \in \overline{\mathcal{G}} \) such that for all \( \mathcal{B} \in \mathcal{P}_f(\mathcal{A}) \) we have \( \bigcap_{B \in \mathcal{B}} B'(q) \in \mathsf{Syn}(\mathcal{F}, \mathcal{G}) \).
\end{definition}

We first show this collectionwise relative piecewise syndetic notion is equivalent to three other characterizations.
One equivalent form, Proposition~\ref{proposition:collectionwise-relative-piecewise-syndetic}(b), shows that our definition could be rewritten in a form more closely matching Definition~\ref{definition:collectionwise-piecewise-syndetic} and \cite[Definition~5.1]{Luperi-Baglini:2023aa}:

\begin{proposition}
\label{proposition:collectionwise-relative-piecewise-syndetic}
    Let \( S \) be a semigroup, let \( \mathcal{F}, \mathcal{G} \) both be proper filters on \( S \), and let \( \mathcal{A} \) be a collection on \(S \).
    The following are equivalent.
    \begin{itemize}
        \item[(a)] 
             \( \mathcal{A} \) is collectionwise piecewise \( (\mathcal{F}, \mathcal{G}) \)-syndetic.

        \item[(b)]
            There exist \(H \in \bigtimes_{(C,\mathcal{B}) \in \mathcal{F}\times\mathcal{P}_f(\mathcal{A})}\mathcal{P}_{f}(C) \) and \( W \in \bigtimes_{(C,\mathcal{B}) \in \mathcal{F}\times\mathcal{P}_f(\mathcal{A})} \mathcal{G} \) such that 
            \[
                \textstyle
                \Bigl\{ w^{-1}\bigl( \bigcup_{h \in H(C, \mathcal{B})} h^{-1}(\bigcap \mathcal{B}) \bigr) : C \in \mathcal{F}, \mathcal{B} \in \mathcal{P}_f(\mathcal{A}), \text{ and } w \in W(C, \mathcal{B}) \Bigr\} \cup \mathcal{G}
            \]
            has the finite intersection property (f.i.p.).

        \item[(c)]
             There exists \( q \in \overline{\mathcal{G}} \) such that for all \( \mathcal{B} \in \mathcal{P}_f(\mathcal{A}) \) we have \( \bigcap \mathcal{B} \in \mathsf{Syn}(\mathcal{F}, \mathsf{Thick}(\mathcal{G}, q)) \).

        \item[(d)] 
            There exists \( q \in \overline{\mathcal{G}} \) such that for all \( p \in \overline{\mathcal{G}} \) there exists \( r \in \overline{\mathcal{F}} \) such that \( \mathcal{A} \subseteq r \cdot p \cdot q \).
    \end{itemize}
\end{proposition}
\begin{proof}
    \textbf{(a) \( \Rightarrow \) (b):}
    Pick \( q \in \overline{\mathcal{G}} \) as guaranteed for \( \mathcal{A} \). 
    For each \( C \in \mathcal{F} \) and \( \mathcal{B} \in \mathcal{P}_f(\mathcal{A}) \), since \( (\bigcap \mathcal{B})'(q) \in \mathsf{Syn}(\mathcal{F}, \mathcal{G}) \), pick \( H(C, \mathcal{B}) \in \mathcal{P}_f(C) \) such that \( \bigcup_{h \in H(C, \mathcal{B})} h^{-1} (\bigcap \mathcal{B})'(q) \in \mathcal{G} \).
    Now put each \( W(C, \mathcal{B}) := \bigcup_{h \in H(C, \mathcal{B})} h^{-1} (\bigcap \mathcal{B})'(q) \) and note \( H \in \bigtimes_{(C,\mathcal{B}) \in \mathcal{F}\times\mathcal{P}_f(\mathcal{A})}\mathcal{P}_{f}(C) \) and \( W \in \bigtimes_{(C,\mathcal{B}) \in \mathcal{F}\times\mathcal{P}_f(\mathcal{A})} \mathcal{G} \).

    To see that f.i.p.~holds let \( C_1, C_2, \ldots, C_m \in \mathcal{F} \), let \( \mathcal{B}_1, \mathcal{B}_2, \ldots, \mathcal{B}_n \in \mathcal{P}_f(\mathcal{A}) \), for each \( i \in \{1, 2, \ldots, m \} \) and \( j \in \{ 1, 2, \ldots, n \} \) let \( W_{i, j} \in \mathcal{P}_f(W(C_i, \mathcal{B}_j)) \), and let \( D \in \mathcal{G} \).
    For each \( (i, j) \in \{1, 2, \ldots, m\} \times \{1, 2, \ldots, n\} \) and each \( w \in W_{i,j} \) we have \( w^{-1}\bigl( \bigcup_{h \in H(C_i, \mathcal{B}_j)} h^{-1}(\bigcap\mathcal{B}_j) \bigr) \in q \) and \( D \in q \).
    Therefore
    \[
         D \cap \bigcap\bigl\{ w^{-1}\bigl(\bigcup_{h \in H(C_i, \mathcal{B}_j)} h^{-1}\bigl(\bigcap\mathcal{B}_j\bigr)\bigr) : (i, j) \in \{1, 2, \ldots, m\} \times  \{1, 2, \ldots, n\} \text{ and } w \in W_{i, j} \bigr\} \in q
    \]

    \smallskip

    \textbf{(b) \( \Rightarrow \) (a):}
    Pick \( H \in \bigtimes_{(C,\mathcal{B}) \in \mathcal{F}\times\mathcal{P}_f(\mathcal{A})}\mathcal{P}_{f}(C) \) and \( W \in \bigtimes_{(C,\mathcal{B}) \in \mathcal{F}\times\mathcal{P}_f(\mathcal{A})} \mathcal{G} \) as guaranteed for \( \mathcal{A} \).
    Pick \( q \in \overline{\mathcal{G}} \) such that for each \( C \in \mathcal{F} \), \( \mathcal{B} \in \mathcal{P}_f(\mathcal{A}) \), and \( w \in W(C, \mathcal{B}) \) we have \( w^{-1} \bigl( \bigcup_{h \in H(C, \mathcal{B})} h^{-1}(\bigcap\mathcal{B}) \bigr) \in q \).
    Then
    \(
        W(C, \mathcal{B}) \subseteq \bigl(\bigcup_{h \in H(C, \mathcal{B})} h^{-1}(\bigcap\mathcal{B})\bigr)'(q)
    \)
    implies 
    \(
         \bigl(\bigcup_{h \in H(C, \mathcal{B})} h^{-1}(\bigcap\mathcal{B})\bigr)'(q) \in \mathcal{G}
    \).
    By Corollary~\ref{corollary:derived-set}(a, iii) and (d) we have \(  \bigl(\bigcup_{h \in H(C, \mathcal{B})} h^{-1}(\bigcap\mathcal{B})\bigr)'(q) =  \bigcup_{h \in H(C, \mathcal{B})} h^{-1}\bigl(\bigcap\mathcal{B}\bigr)'(q) \) and so, by Definition~\ref{definition:relative-syndetic-thick}(a), we have \( (\bigcap\mathcal{B})'(q) \in \mathsf{Syn}(\mathcal{F}, \mathcal{G}) \).
  
    \smallskip

    \textbf{(a) \( \Leftrightarrow \) (c):} This follows from Proposition~\ref{proposition:derived-set-relative-syndetic}.

    \smallskip

    \textbf{(c) \( \Rightarrow \) (d):}
    Pick \( q \in \overline{\mathcal{G}} \) as guaranteed for \( \mathcal{A} \).
    Then by Corollary~\ref{corollary:relative-syndetic-thick} we have for all \( \mathcal{B} \in \mathcal{P}_f(\mathcal{A}) \) and all \( p \in \overline{\mathcal{G}} \) that \( \bigcap\mathcal{B} \in \mathcal{F}^* \cdot p \cdot q \). 
    Since each \( \mathcal{F}^* \cdot p \cdot q \) is a grill (Corollary~\ref{corollary:derived-set}(d, iii)) it follows that there exists \( r \in \overline{\mathcal{F}} \) with \( \mathcal{A} \subseteq r \cdot p \cdot q \).

    \smallskip
    \textbf{(d) \( \Rightarrow \) (c):}
    Pick \( q \in \overline{\mathcal{G}} \) as guaranteed and let \( p \in \overline{\mathcal{G}} \).
    Pick \( r \in \overline{\mathcal{F}} \) as guaranteed.
    Then for all \( \mathcal{B} \in \mathcal{P}_f(\mathcal{A}) \) we have \( \bigcap\mathcal{B} \in r \cdot p \cdot q \), which implies \( \bigcap\mathcal{B} \in \mathcal{F}^* \cdot p \cdot q \). 
    Hence, from Corollary~\ref{corollary:relative-syndetic-thick} it follows that \( \bigcap\mathcal{B} \in \mathsf{Syn}(\mathcal{F}, \mathsf{Thick}(\mathcal{G}, q)) \).
    
\end{proof}

Generalizing \cite[Theorem~5.2]{Luperi-Baglini:2023aa} we show collectionwise piecewise \( (\mathcal{F}, \mathcal{G}) \)-syndetic  characterizes those collections contained in ultrafilters of \( K(\mathcal{F}, \mathcal{G}) \): 
\begin{theorem}
\label{theorem:collectionwise-relative-piecewise-syndetic}
	Let \( S \) be a semigroup, let	\( \mathcal{F}, \mathcal{G} \) both be proper filters on \( S \) with \( \mathcal{F} \subseteq \mathsf{Syn}(\mathcal{F}^*, \mathcal{G}) \) and \( \mathcal{G} \) a product filter, and let \( \mathcal{A} \subseteq \mathcal{P}(S) \) be a collection.
	Then \( \mathcal{A} \) is collectionwise piecewise \( (\mathcal{F}, \mathcal{G}) \)-syndetic if and only if there exists \( p \in K(\mathcal{F}, \mathcal{G}) \) with \( \mathcal{A} \subseteq p \).
\end{theorem}
\begin{proof}
	\textbf{(\( \Rightarrow \)):}
	Pick \( q \in \overline{\mathcal{G}} \) as guaranteed for \( \mathcal{A} \).
	Fix any \( e \in E(K(\overline{\mathcal{G}})) \).
	By Corollary~\ref{corollary:relative-syndetic-thick}(a), for all \( \mathcal{B} \in \mathcal{P}_f(\mathcal{A}) \) we have \( \bigcap_{B \in \mathcal{B}} B'(q) \in \mathcal{F}^* \cdot e \).
	Now by Corollary~\ref{corollary:derived-set}(d, iii) \( \mathcal{F}^* \cdot e \) is a grill on \( S \) and so it follows there exists \( r \in \overline{\mathcal{F}} \) with \( \mathcal{A} \subseteq r \cdot e \cdot q \).
	Since \( K(\overline{\mathcal{G}}) \) is an ideal in \( \overline{\mathcal{G}} \) we have \( e \cdot q \in K(\overline{\mathcal{G}}) \) and so \( r \cdot e \cdot q \in K(\mathcal{F}, \mathcal{G}) \).
	
	\smallskip
	
	\textbf{(\( \Leftarrow \)):}
	Pick \( p \in K(\mathcal{F}, \mathcal{G}) \) as guaranteed for \( \mathcal{A} \).
	By Theorem~\ref{theorem:relative-kernel} pick \( e \in E(K(\overline{\mathcal{G}})) \) with \( \{ A'(e) : A \in p \} \subseteq \mathsf{Syn}(\mathcal{F}, \mathcal{G}) \).
	For \( \mathcal{B} \in \mathcal{P}_f(\mathcal{A}) \) we have \( \bigcap_{B \in \mathcal{B}} B'(e) = \bigl( \bigcap_{B \in \mathcal{B}} B \bigr)'(e) \in  \mathsf{Syn}(\mathcal{F}, \mathcal{G}) \), where the equality follows from Proposition~\ref{proposition:derived-set}(a, iii).
	Hence \( \mathcal{A} \) is collectionwise piecewise \( (\mathcal{F}, \mathcal{G}) \)-syndetic.
\end{proof}

In particular, we note we have also characterized collectionwise piecewise (\( \mathcal{F} \)-)syndetic sets:
\begin{corollary}
\label{corollary:collectionwise-relative-piecewise-syndetic}
    Let \( S \) be a semigroup and let \( \mathcal{A} \) be a collection on \(S \).
    \begin{itemize}
        \item[(a)]
            If \( \mathcal{F} \) is a proper product filter, then \( \mathcal{A} \) is collectionwise piecewise \( (\mathcal{F}, \mathcal{F}) \)-syndetic if and only if \( \mathcal{A} \) is collectionwise piecewise \( \mathcal{F} \)-syndetic.

        \item[(b)]
            \( \mathcal{A} \) is collectionwise piecewise \( (\{S\}, \{S\}) \)-syndetic if and only if \( \mathcal{A} \) is collectionwise piecewise syndetic
    \end{itemize}
\end{corollary}
\begin{proof}
    We have \( \overline{\mathcal{F}} \) has a smallest ideal and, by Corollary~\ref{corollary:relative-kernel}(c) we have \( K(\mathcal{F}, \mathcal{F}) = K(\overline{\mathcal{F}}) \).

    \smallskip
    
    \textbf{(a):}
    By Theorem~\ref{theorem:collectionwise-relative-piecewise-syndetic} we have \( \mathcal{A} \subseteq \mathcal{P}(S) \) is collectionwise piecewise \( (\mathcal{F}, \mathcal{F}) \)-syndetic if and only if there exists \( p \in  K(\mathcal{F}, \mathcal{F}) = K(\overline{\mathcal{F}}) \) with \( \mathcal{A} \subseteq p \), and, by \cite[Theorem~5.2]{Luperi-Baglini:2023aa}, we extend this equivalence to \( \mathcal{A} \) is collectionwise piecewise \( \mathcal{F} \)-syndetic.

    \smallskip

    \textbf{(b):}
    Apply statement (a) with \( \mathcal{F} = \{S\} \).
\end{proof}

Furstenberg \cite[Ch.~8]{Furstenberg1981a} (dynamically) defined the notion of a central set in \( (\mathbb{N}, +) \) and proved his powerful combinatorial central sets theorem.
Bergelson and Hindman, with the assistance of B.~Weiss, then characterized central sets in arbitrary semigroups as members of idempotents in \( K(\beta S) \) \cite{Bergelson:1990aa} and proved a central sets theorem for semigroups.
This algebraic point-of-view of central sets has enabled the proof of stronger versions of the central sets theorem for semigroups, such as in De, Hindman, Strauss \cite{De:2008aa}.
Hindman's survey articles on central sets and other notions of size \cite{Hindman:2020aa, Hindman2019a} document the extensive history and applications of central sets in dynamics, algebra, and Ramsey theory.

Goswami and Poddar defined a relative notion of central sets and also proved a relative central sets theorem \cite{Goswami:2022aa}, building on earlier work of De, Hindman, and Strauss proving a relative central sets theorem for subsemigroups of \( \beta S \) where \( S \) is a commutative semigroup \cite{De:2009aa}.
If \( \mathcal{F} \) is a proper product filter, then \( A \subseteq S \) is \define{\( \mathcal{F} \)-central} if and only if there exists an idempotent \( e \in K(\overline{\mathcal{F}}) \) with \( A \in e \).
Our goal now is to generalize this relative notion of central sets to account for two filters \( \mathcal{F}, \mathcal{G} \), however, under the assumptions of Theorem~\ref{theorem:relative-kernel} it is not clear \emph{if} \( K(\mathcal{F}, \mathcal{G}) \) possess idempotents.
The following result provides a sufficient condition ensuring \( K(\mathcal{F}, \mathcal{G}) \) does have idempotents:
\begin{theorem}
\label{theorem:relative-kernel-has-idempotents}
    Let \( S \) be a semigroup, let \( \mathcal{F}, \mathcal{G} \) both be proper product filters satisfying the condition \( \mathcal{F} \subseteq \mathsf{Syn}(\mathcal{F}^*, \mathcal{G}) \).
    Then  \( K(\mathcal{F}, \mathcal{G}) \) is a subsemigroup of \( \beta S \) with  \( E\bigl( K(\mathcal{F}, \mathcal{G}) \bigr) \ne \emptyset \).    

\end{theorem}
\begin{proof}
    From Theorem~\ref{theorem:relative-kernel} we have
    \(
        K(\mathcal{F}, \mathcal{G}) = \bigcup_{e \in K(\overline{\mathcal{G}})} \overline{\mathcal{F}} \cdot e
    \).
    Each \( \overline{\mathcal{F}} \cdot e \) is a closed (recall  \( p \mapsto p \cdot e \) is continuous) left ideal of \( \overline{\mathcal{F}} \).
    Hence \(K(\mathcal{F}, \mathcal{G}) \) is a subsemigroup of \( \overline{\mathcal{F}} \) and \( E\bigl( K(\mathcal{F}, \mathcal{G}) \bigr) \ne \emptyset \) follows because each \( E(\overline{\mathcal{F}} \cdot e) \ne \emptyset \).    

\end{proof}
Note if we replace the assumption `\( \mathcal{F} \subseteq \mathsf{Syn}(\mathcal{F}^*, \mathcal{G}) \)' by `\( \mathcal{G} \subseteq \mathsf{Syn}(\mathcal{G}^*, \mathcal{F}) \)' in Theorem~\ref{theorem:relative-kernel-has-idempotents}, while keeping the other two assumptions, then similarly it follows that  \( K(\mathcal{F}, \mathcal{G}) \) is a subsemigroup with idempotents.
With either choice, we could then define \( A \) as ``\( (\mathcal{F}, \mathcal{G}) \)-central'' if and only if \( \overline{A} \cap E\bigl(K(\mathcal{F}, \mathcal{G})\bigr) \ne \emptyset \). 
Such a definition would show if \( A_1 \cup A_2 \) is ``\( (\mathcal{F}, \mathcal{G}) \)-central'' then either \( A_1 \) or \( A_2 \) is also.
If \( \mathcal{F} = \mathcal{G} \), then, by Corollary~\ref{corollary:relative-kernel}(c), this would reduce to Goswami and Poddar's definition of \( \mathcal{F} \)-central.

Instead of taking the above as our notion of relative central, we follow a recent point-of-view of Glasscock and Le \cite{Glasscock:2024aa} that, in particular, characterizes central sets as members of certain idempotent filters (not necessarily ultrafilters).
Then \( A \subseteq S \) is \define{central} if and only if there exists a proper idempotent filter \( \mathcal{F} \) on \( S \) with \( A \in \mathcal{F} \) and \( \mathcal{F} \) is collectionwise piecewise syndetic.
(This will follow as a special case of Theorem~\ref{theorem:relative-central-is-partition-regular}.)

\begin{definition}
\label{definition:relative-central}
    Let \( S \) be a semigroup and let \( \mathcal{F}, \mathcal{G} \) both be proper filters on \( S \).
    Call \( A \subseteq S \) an \define{\( (\mathcal{F}, \mathcal{G}) \)-central} set if and only if there exists a idempotent proper filter \( \mathcal{H} \) with \( A \in \mathcal{H} \) and \( \mathcal{H} \) is collectionwise piecewise \( (\mathcal{F}, \mathcal{G}) \)-syndetic.
\end{definition}

Observe \( S \) is \( (\mathcal{F}, \mathcal{G}) \)-central, however with this definition we don't know if this particular notion has the Ramsey property (see, for instance, Definition~\ref{definition:filter-grill}(b)): 
\begin{question}
\label{question:partition-regular}
    Let \( \mathcal{F}, \mathcal{G} \) both be proper filters on a semigroup \( S \) and let \( A_1, A_2 \subseteq S \).
    If \( A_1 \cup A_2 \) is \( (\mathcal{F}, \mathcal{G}) \)-central, does there exists \( i \in \{1, 2\} \) such that \( A_i \) is \( (\mathcal{F}, \mathcal{G}) \)-central?
\end{question}
We'll end this section by showing the answer to Question~\ref{question:partition-regular} is ``yes'' under the assumptions of Theorem~\ref{theorem:relative-kernel-has-idempotents}, but we don't know the answer under weaker assumptions on \( \mathcal{F} \) and \( \mathcal{G} \).
Similar to the proof of \cite[Theorem~5.8]{Luperi-Baglini:2023aa}, our proof follows the Hindman, Maleki, and Strauss \cite[Theorem~3.8]{Hindman:1996aa} combinatorial characterizations of elements of minimal idempotents (which are the central sets).

\begin{theorem}
\label{theorem:relative-central-is-partition-regular}
    Let \( S \) be a semigroup, let \( \mathcal{F}, \mathcal{G} \) both be proper product filters satisfying the condition \( \mathcal{F} \subseteq \mathsf{Syn}(\mathcal{F}^*, \mathcal{G}) \), and let \( A \subseteq S \).
    The following statements are equivalent.
    \begin{itemize}
        \item[(a)]
            \( A \) is \( (\mathcal{F}, \mathcal{G}) \)-central.
            
        \item[(b)]
            \( \overline{A} \cap E\bigl( K(\mathcal{F}, \mathcal{G}) \bigr) \ne \emptyset \).
    \end{itemize}
\end{theorem}
\begin{proof}
    \textbf{(a) \( \Rightarrow \) (b):}
    Pick a proper filter \( \mathcal{H} \) as guaranteed for \( A \).
    By Theorems~\ref{theorem:collectionwise-relative-piecewise-syndetic} and \ref{theorem:relative-kernel} it follows that there exists \( e \in E(K(\overline{\mathcal{G}})) \) with \( \overline{\mathcal{H}} \cap \overline{\mathcal{F}} \cdot e \ne \emptyset \).
    From the proof of Theorem~\ref{theorem:relative-kernel-has-idempotents} we have \( \overline{\mathcal{F}} \cdot e \) is a closed subsemigroup of \( \beta S \) (in fact, a left ideal of \( \overline{\mathcal{F}} \)).
    Hence it suffices to show \( \overline{\mathcal{H}} \) is closed subsemigroup too.
    (For then \( \overline{\mathcal{H}} \cap \overline{\mathcal{F}} \cdot e \) is a closed subsemigroup that contains an idempotent which is in \( K(\mathcal{F}, \mathcal{G}) \).)
    Since \( \mathcal{H} \) is an idempotent filter we have \( \mathcal{H} \subseteq \mathcal{H} \cdot \mathcal{H} = \mathsf{Thick}(\mathcal{H}, \mathcal{H}^*) \subseteq \mathsf{Syn}(\mathcal{H}^*, \mathcal{H}) \), where the equality follows from Corollary~\ref{corollary:relative-syndetic-thick}(b) and the second inclusion follows from Corollary~\ref{corollary:relative-syndetic-thick}(a) and Proposition~\ref{proposition:derived-set}(c, i).

    \smallskip

    \textbf{(b) \( \Rightarrow \) (a):}
    Pick \( e \in \overline{A} \cap E\bigl( K(\mathcal{F}, \mathcal{G}) \bigr) \) as guaranteed. 
    Then \( A \in e \), \( e \) is an idempotent (ultra)filter, and, by Theorem~\ref{theorem:collectionwise-relative-piecewise-syndetic}, \( e \) is collectionwise piecewise \( (\mathcal{F}, \mathcal{G}) \)-syndetic.
\end{proof}

\begin{corollary}
\label{corollary:relative-central-is-partition-regular}
     Let \( S \) be a semigroup, let \( \mathcal{F}, \mathcal{G} \) both be proper product filters satisfying the condition \( \mathcal{F} \subseteq \mathsf{Syn}(\mathcal{F}^*, \mathcal{G}) \).
     Then \( \mathsf{Cen}(\mathcal{F}, \mathcal{G}) := \{ A \subseteq S : A \text{ is } (\mathcal{F}, \mathcal{G})\text{-central} \} \) is a proper grill.
\end{corollary}
\begin{proof}
    Note \( \mathsf{Cen}(\mathcal{F}, \mathcal{G}) \) is a proper stack just follows from the assumption that  \( \mathcal{F}, \mathcal{G} \) are proper filters.
    If \( A_1 \cup A_2 \in \mathsf{Cen}(\mathcal{F}, \mathcal{G}) \), then, by Theorem~\ref{theorem:relative-central-is-partition-regular}, pick \( e \in \overline{A_1 \cup A_2} \cap E\bigl( K(\mathcal{F}, \mathcal{G}) \bigr) \). 
    Pick \( i \in \{1, 2\} \) with \( A_i \in e \subseteq \mathsf{Cen}(\mathcal{F}, \mathcal{G}) \).
\end{proof}

\subsection*{Acknowledgments}
\label{section:acknowledgments}
We thank  Lorenzo Luperi Baglini, Daniel Glasscock, Sayan Goswami, Neil Hindman,  Anh Le, Sourav Kanti Patra, and Md.~Moid Shaikh for reading an earlier draft of this article and providing helpful feedback.
Research for the all the authors was supported by the American Institute of Mathematics (AIM) Structured Quartet Research Ensemble (SQuaRE) \emph{Relative notions of size and the Stone--Čech compactification} project. 
The authors thank AIM for their support and conducive space for meeting and working on mathematics during Autumn 2022 and Summer 2024.

\printbibliography

\end{document}